\documentclass{article}
\usepackage{fancyhdr}
\usepackage{amsmath, amssymb, latexsym, amscd, amsthm,amsfonts,amstext}
\usepackage{subfigure}
\usepackage{enumerate}
\usepackage{xcolor}
\usepackage{textcomp}
\usepackage{mathtools}

\usepackage{graphicx}
\newtheorem{theorem}{Theorem}[section]
\newtheorem{lemma}[theorem]{Lemma}

\numberwithin{equation}{section}
\newtheorem{definition}[theorem]{Definition}
\newtheorem{remark}[theorem]{Remark}

\numberwithin{equation}{section}

\usepackage[mathscr]{eucal}
\usepackage{graphicx}
\usepackage{subfigure}
\usepackage{cite}
\usepackage{stfloats}
\usepackage{fancyhdr}
\usepackage{enumerate}
\usepackage{xcolor}

\allowdisplaybreaks

\title{The equivalent refraction index for the acoustic scattering by many small obstacles: with error estimates.} 
\author{Bashir Ahmad\thanks{Nonlinear Analysis and Applied Mathematics Research Group (NAAM), Department of  Mathematics,
Faculty of Sciences,
King Abdulaziz University,
P.O. Box 80203,
Jeddah 21589,
Saudi Arabia.
(Email: bashirahmad\_qau@yahoo.com).} 
\and  
 Durga Prasad Challa\footnote{C\lowercase{orresponding author}: D\lowercase{urga} P\lowercase{rasad} C\lowercase{halla}.} 
\thanks{Department of mathematics, 
Inha university,  Incheon 402-751, S. Korea.
(Email: durga.challa@inha.ac.kr).
}
\and Mokhtar Kirane\footnotemark[1]
\thanks{Laboratoire de Math\'ematiques,
P\^ole Sciences et Technologies, 
Universit\'e de La Rochelle, 
Avenue Michel Cr\'epeau
17042, La Rochelle Cedex, France, and Nonlinear Analysis and Applied Mathematics Research Group (NAAM), Department of  Mathematics,
Faculty of Sciences,
King Abdulaziz University,
P.O. Box 80203,
Jeddah 21589,
Saudi Arabia. (Email: mokhtar.kirane@univ-lr.fr).}
\and  Mourad Sini\thanks{RICAM, Austrian Academy of Sciences,
Altenbergerstrasse 69, A-4040, Linz, Austria.
(Email: mourad.sini@oeaw.ac.at).
}
}
 
\begin{document}
\graphicspath{{Figures-eps/}}
 \maketitle
\begin{abstract}
 Let $M$ be the number of bounded and Lipschitz regular obstacles $D_j, j:=1, ..., M$ having a maximum radius $a$, $a<<1$, located in a bounded domain $\Omega$ of $\mathbb{R}^3$. 
 We are concerned with the acoustic scattering problem with a very large number of obstacles, as $M:=M(a):=O(a^{-1})$, $a\rightarrow 0$, when they are arbitrarily distributed in $\Omega$
 with a minimum distance between them of the order $d:=d(a):=O(a^t)$ with $t$ in an appropriate range.
 We show that the acoustic farfields corresponding to the scattered waves by this collection of obstacles, taken to be soft obstacles, converge uniformly in terms of the incident as well the propagation directions, 
 to the one corresponding to an acoustic refraction index as $a\rightarrow 0$. This refraction index is given as a product of two coefficients $\bold{C}$ and $K$, where the first one 
 is related to the geometry of the obstacles (precisely their capacitance) and the second one is related to the local distribution of these obstacles. 
 In addition, we provide explicit error estimates, in terms of $a$, in the case when the obstacles are locally the same (i.e. have the same capacitance, or the coefficient 
 $\bold{C}$ is piecewise constant) in $\Omega$ and the coefficient $K$ is H$\ddot{\mbox{o}}$lder continuous. These approximations can be applied, in particular, to the theory of acoustic materials 
 for the design of refraction indices by perforation using either the geometry of the holes, i.e. the coefficient $\bold{C}$, or their local distribution in a given domain $\Omega$, 
 i.e. the coefficient $K$.
\end{abstract}
 \textbf{Keywords}: Acoustic scattering, Multiple scattering,  Efective medium.

\pagestyle{myheadings}
 \thispagestyle{plain}
 \markboth{  B. Ahmad, D. P. Challa, M. Kirane and M. Sini }{The equivalent medium for acoustic scattering by many small obstacles}

\section{Introduction and statement of the results}\label{Introduction-smallac-sdlp}

 Let $B_1, B_2,\dots, B_M$ be $M$ open, bounded and simply connected sets in $\mathbb{R}^3$ with Lipschitz boundaries containing the origin.
We assume that the Lipschitz constants of $B_j$, $j=1,..., M$ are uniformly bounded.  
We set $D_m:=\epsilon B_m+z_m$ to be the small bodies characterized by the parameter 
$\epsilon>0$ and the locations $z_m\in \mathbb{R}^3$, $m=1,\dots,M$. 
We denote by  $U^{s}$ the acoustic field scattered by the $M$ small and soft bodies $D_m\subset \mathbb{R}^{3}$ due to 
the incident plane wave $U^{i}(x,\theta):=e^{ikx\cdot\theta}$, 
with the incident direction $\theta \in \mathbb{S}^2$, with $\mathbb{S}^2$ being the unit sphere. Hence the total field $U^{t}:=U^{i}+U^{s}$ satisfies the following exterior Dirichlet problem of the acoustic waves
\begin{equation}
(\Delta + \kappa^{2})U^{t}=0 \mbox{ in }\mathbb{R}^{3}\backslash \left(\mathop{\cup}_{m=1}^M \bar{D}_m\right),\label{acimpoenetrable}
\end{equation}
\begin{equation}
U^{t}|_{\partial D_m}=0,\, 1\leq m \leq M, \label{acgoverningsupport}  
\end{equation}
\begin{equation}
\frac{\partial U^{s}}{\partial |x|}-i\kappa U^{s}=o\left(\frac{1}{|x|}\right), |x|\rightarrow\infty, ~(\text{S.R.C}) \label{radiationc}
\end{equation}
where  $\kappa>0$ is the wave number, $\kappa=2\pi\slash \lambda$, $\lambda$ is the wave length and S.R.C stands for the Sommerfield radiation condition.  The scattering problem 
(\ref{acimpoenetrable}-\ref{radiationc}) is well posed in appropriate spaces, see \cite{C-K:1998, Mclean:2000} for instance, and the scattered field $U^s(x, \theta)$ has the following asymptotic expansion:
\begin{equation}\label{far-field}
 U^s(x, \theta)=\frac{e^{i \kappa |x|}}{|x|}U^{\infty}(\hat{x}, \theta) + O(|x|^{-2}), \quad |x|
\rightarrow \infty,
\end{equation}
with $\hat{x}:=\frac{x}{\vert x\vert}$, where the function
$U^{\infty}(\hat{x}, \theta)$ for $(\hat{x}, \theta)\in \mathbb{S}^{2} \times \mathbb{S}^{2}$  is called the far-field pattern.
We recall that the fundamental solution $\Phi_\kappa(x,y)$ of the Helmholtz equation in $\mathbb{R}^3$
with the fixed wave number $\kappa$ is given by $\Phi_\kappa(x,y):=\frac{e^{i\kappa|x-y|}}{4\pi|x-y|},\quad \text{for all } x,y\in\mathbb{R}^3$.
\bigskip

\begin{definition} 
\label{Def1}
We define 
\begin{enumerate}
 \item $a$ as the maximum among the diameters, $diam$, of the small bodies $D_m$, i.e.
\begin{equation}\label{def-a} 
a:=\max\limits_{1\leq m\leq M } diam (D_m) ~~\big[=\epsilon \max\limits_{1\leq m\leq M } diam (B_m)\big],
\end{equation}

 \item  $d$ as the minimum distance  between the small bodies $\{D_1,D_2,\dots,D_m\}$, i.e.
\begin{equation}\label{def-d}
d:=\min\limits_{\substack{m\neq j\\1\leq m,j\leq M }} d_{mj},
\end{equation}
$\text{where}\,d_{mj}:=dist(D_m, D_j)$. We assume that
\begin{equation}\label{def-dmax}
0\,<\,d\,\leq\,d_{\max},
\end{equation}
and $d_{\max}$ is given.
\item $\kappa_{\max}$ as the upper bound of the used wave numbers, i.e. $\kappa\in[0,\,\kappa_{\max}]$.
\end{enumerate} 
\end{definition}
\bigskip

We assume that $D_m=\epsilon{B}_m+z_m, m=1,\dots,M$, with the same diameter $a$, are non-flat Lipschitz obstacles, i.e. $D_m$'s are Lipschitz obstacles and there exist constants $t_m \in (0, 1]$ such that

\begin{equation}\label{non-flat-condition}
 B^{3}_{t_m\frac{a}{2}}(z_m)\subset\,D_m\subset\,B^{3}_{\frac{a}{2}}(z_m),
 \end{equation}
 where $t_m $ are assumed to be uniformly bounded from below by a positive constant.
\bigskip

In a recent work \cite{C-S:2014}, we have shown that there exist two positive constants $a_0$ and $c_0$ depending only on the 
Lipschitz character of $B_m,m=1,\dots,M$, $d_{\max}$ and $\kappa_{\max}$ such that
if 
\begin{equation} \label{conditions}
a \leq a_0 ~~ \mbox{and} ~~ \sqrt{M-1}\frac{a}{d}\leq c_0
\end{equation} 
then the far-field pattern $U^\infty(\hat{x},\theta)$ has the following asymptotic expansion

\begin{eqnarray}\label{x oustdie1 D_m farmain-recent-near}
 \hspace{-1cm}U^\infty(\hat{x},\theta)
&=&\sum_{m=1}^{M}e^{-i\kappa\hat{x}\cdot z_m}Q_m\nonumber\\
&&+O\left(M\left[a^2+\frac{a^3}{d^{5-3\alpha}}+\frac{ a^4}{d^{9-6\alpha}}\right]+M(M-1)\left[\frac{a^3}{d^{2\alpha}}+\frac{a^4}{d^{4-\alpha}}+\frac{a^4}{d^{5-2\alpha}}\right]
+M(M-1)^2\frac{a^4}{d^{3\alpha}}\right)
 \end{eqnarray}
uniformly in $\hat{x}$ and $\theta$ in $\mathbb{S}^2$, where the parameter $\alpha$, $0< \alpha \leq 1$, is related to the number of obstacles localed 'near' a given obstacle, see 
\cite{C-S:2014} and explicit formulation. The coefficients $Q_m$, $m=1,..., M,$ are the solutions of the following linear algebraic system
\begin{eqnarray}\label{fracqcfracmain}
 Q_m +\sum_{\substack{j=1 \\ j\neq m}}^{M}C_m \Phi_\kappa(z_m,z_j)Q_j&=&-C_mU^{i}(z_m, \theta),~~
\end{eqnarray}
for $ m=1,..., M,$ with $C_m:=\int_{\partial D_m}\sigma_m(s)ds$ and $\sigma_{m}$ is 
the solution of the integral equation of the first kind
\begin{eqnarray}\label{barqcimsurfacefrm1main}
\int_{\partial D_m}\frac{\sigma_{m} (s)}{4\pi|t-s|}ds&=&1,~ t\in \partial D_m.
\end{eqnarray}
The algebraic system \eqref{fracqcfracmain} is invertible under the conditions:
\footnote{If $\Omega$ is a domain containing the small bodies, and $diam(\Omega)$ denotes its diameter, 
then one example for the validity of the second condition in \eqref{invertibilityconditionsmainthm} is $diam(\Omega)< \frac{\pi}{2\kappa}$.}
\begin{eqnarray}\label{invertibilityconditionsmainthm}
 \frac{a}{d}\leq c_1 \text{ and } 
 \min_{j\neq m}\cos(\kappa \vert z_j- z_m\vert)\geq 0,
\end{eqnarray}
where $c_1$ depends only on the Lipschitz character of the obstacles $B_j$, $j=1, ..., M$. 

\bigskip

The formula (\ref{x oustdie1 D_m farmain-recent-near}) says that the farfields corresponding to $M$ obstacles can be approximated by the expression $\sum_{m=1}^{M}e^{-i\kappa\hat{x}\cdot z_m}Q_m$, that we call 
the Foldy-Lax field since it is reminiscent to the field generated by a collection of point-like scatterers \cite{LLF:PR1945, Lax-M:RMP1951}, see also the monograph 
\cite{Martin:2006}. 
Hence if we have a reasonably large number of obstacles, we can reduce the scattering problem 
to an inversion of the an algebraic system, i.e.  (\ref{fracqcfracmain}). In this paper, we are concerned with the case where we have an extremely large number of 
obstacles of the form  $M:=M(a):=O(a^{-s})$ with  $s>0$
and the minimum distance $d:=d(a):=O(a^t)$ with  $t>0$. In this case, the asymptotic expansion \eqref{x oustdie1 D_m farmain-recent-near} can be rewritten as 
\begin{eqnarray}\label{x oustdie1 D_m farmain-recent**}
U^\infty(\hat{x},\theta)
\hspace{-.05cm}&=&\hspace{-.1cm}\sum_{m=1}^{M}e^{-i\kappa\hat{x}\cdot z_m}Q_m\hspace{-.03cm}+\hspace{-.03cm}O\left(a^{2-s}\hspace{-.03cm}+\hspace{-.03cm}a^{3-s-5t+3t\alpha}\hspace{-.03cm}+
\hspace{-.03cm}a^{4-s-9t+6t\alpha}\hspace{-.03cm}+\hspace{-.03cm}a^{3-2s-2t\alpha}\hspace{-.03cm}+\hspace{-.03cm}a^{4-3s-3t\alpha}\hspace{-.03cm}+\hspace{-.03cm}a^{4-2s-5t+2t\alpha}\right).\nonumber\\
 \end{eqnarray}
As the diameter $a$ tends to zero the error term tends to zero for $t$ and $s$ such that 
\begin{equation}\label{general-condion-s-t}
0<t<1\; \mbox{ and } 0<s<\min\{2(1-t),\,\frac{7-5t}{4},\,\frac{12-9t}{7},\frac{20-15t}{12},\frac{4}{3}-t\alpha\}.
\end{equation}
Observe that we have the upper bound
\begin{equation}
 \vert \sum_{m=1}^{M}e^{-i\kappa\hat{x}\cdot z_m}Q_m\vert \leq M\sup_{m=1, ..., M}\vert Q_m\vert=O(a^{1-s}) 
\end{equation}
since $Q_m\approx a$, see \cite{C-S:2014}. Hence if the number of obstacles is $M:=M(a):=a^{-s}, \; s<1$ and $t$ satisfies (\ref{general-condion-s-t}), $a\rightarrow 0$, then from (\ref{x oustdie1 D_m farmain-recent**}), we deduce that
\begin{equation}\label{s-smaller-1}
 U^\infty(\hat{x},\theta)\rightarrow 0, \mbox{ as } a \rightarrow 0, \mbox{ uniformly in terms of } \theta \mbox{ and } \hat{x} \mbox{ in } \mathbb{S}^2.
\end{equation}
This means that this collection of obstacles has no effect on the homogeneous medium as $a \rightarrow 0$. 
The main concern of this paper is to consider the case when $s=1$. To start, let $\Omega$ be a bounded domain, say of unit volume, containing the obstacles $D_m, m=1, ..., M$. 
We shall divide $\Omega$ into $[a^{-1}]$ sub-domains $\Omega_m,\; m=1, ..., [a^{-1}]$ such that each 
$\Omega_m$ contains $D_m$, with $z_m \in \Omega_m$ as its center, and some of the other $D_j$'s. It is natural then to assume that the number of obstacles in $\Omega_m$, 
for $m=1, ..., [a^{-1}]$, to be uniformly bounded in terms of $m$. To describe correctly this number of obstacles, we introduce 
 $K: \mathbb{R}^3\rightarrow \mathbb{R}$ as a positive continuous and bounded function. 
 Let each $\Omega_m$, $m\in \mathbb{N}$, be a cube such that $\Omega_m \cap \Omega$ (which we denote also by $\Omega_m$) is of volume $a\frac{[K(z_m)+1]}{K(z_m)+1}$
 and contains $[K(z_m) +1]$ obstacles (where $[a]$ stands for the entire part of $a\in \mathbb{R}$). 
We set $K_{max}:=\sup_{z_m}(K(z_m)+1)$, hence $M=\sum^{[a^{-1}]}_{j=1}[K(z_m)+1]\leq K_{max}[a^{-1}]=O(a^{-1})$. 
\begin{figure}[htp]\label{distribution-obstacles}
\centering
\includegraphics[width=6cm,height =5cm,natwidth=610,natheight=642]{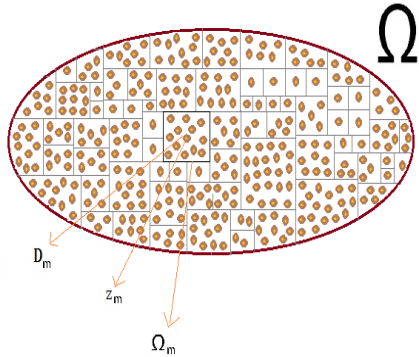}
\caption{An example on how the obstacles are distributed in $\Omega$.}
\end{figure}

\begin{remark}\label{distribution}
 We see that $\bigcup^{[a^{-1}]}_{m=1} \Omega_m \subset \Omega$. Hence $Vol(\bigcup^{[a^{-1}]}_{m=1}\Omega_m)=\sum^{[a^{-1}]}_{m=1}\frac{[K(z_m)+1]}{K(z_m)+1} a \rightarrow 
 \int_\Omega \frac{[K(z)+1]}{K(z)+1} dz$ as $a\rightarrow 0$. Then, as $a\rightarrow 0$, we have $\bigcup^{[a^{-1}]}_{m=1} \Omega_m \varsubsetneq \Omega$ if $K$ is not a function with 
 entire values! In this case, we might not fill in fully $\Omega$. To do it, one needs, for instance, to add $\frac{1}{L}[a^{-1}]$ sub-domains, of the form $\Omega_m$, 
 with an appropriate integer $L$ depending on $\Omega \setminus {\lim_{a\rightarrow 0}\bigcup^{[a^{-1}]}_{m=1}\Omega_m}$. 
 To keep the presentation simple, we take $\Omega:=\lim_{a\rightarrow 0}\bigcup^{[a^{-1}]}_{m=1}\Omega_m$.
\end{remark}
\bigskip

We prove the following result:
\bigskip

\begin{theorem}\label{equivalent-medimu}
Let the small obstacles be distributed in a bounded domain $\Omega$, say of unit volume, with their number $M:=M(a):=O(a^{-1})$ and their minimum distance 
$d:=d(a):=a^{t}$,\; $\frac{1}{3}\leq t < \frac{5}{12}$, as $a\rightarrow 0$, as described above.
\begin{enumerate}
 \item If the obstacles are distributed as follows: $\Omega=\cup^N_{j=1}E_j$, with $N$ fixed, and each sub-domain $\Omega_m$ included $E_j$ contains $K(z_j)+1$ obstacles having the same shape (actually the same capacitance), then we have the asymptotic expansion:
 \begin{equation}\label{A}
  U^\infty(\hat{x},\theta) = U_N^\infty(\hat{x},\theta) +O(a^{\frac{1}{3}-\frac{4}{5}t}), \; a \rightarrow 0, \mbox{ uniformly in terms of } \theta \mbox{ and } \hat{x} \mbox{ in } \mathbb{S}^2,
 \end{equation}
where $U_N^\infty(\hat{x},\theta)$ is the farfield corresponding to the scattering problem

\begin{equation}
(\Delta + \kappa^{2}-K_N\bold{C}_N)U_N^{t}=0 \mbox{ in }\mathbb{R}^{3},\label{A-1}
\end{equation}
\begin{equation}
U_N^{t}=U_N^s +e^{i\kappa x\cdot \theta},  
\end{equation}
\begin{equation}
\frac{\partial U_N^{s}}{\partial |x|}-i\kappa U_N^{s}=o\left(\frac{1}{|x|}\right), |x|\rightarrow\infty, \label{radiationc-A-1}
\end{equation}
where $\bold{C}_N=0, \; \mbox{in } \mathbb{R}^{3} \setminus{\overline \Omega}$ and $\bold{C}_N=\overline{C}_j, \; \mbox{ in } E_j$, $j=1, ..., N$ and 
$\overline{C}_j\; a$ is the capacitance of the (same) obstacles included in $E_j$. Similarly $K_N=0, \; \mbox{in } \mathbb{R}^{3} \setminus{\overline \Omega}$ and 
$K_N=K(z_j)+1, \; \mbox{ in } E_j$, $j=1, ..., N$.

 \item If the obstacles are distributed arbitrary in $\Omega$, i.e. with different capacitances, then there exists a potential $\bold{C}_0 \in \cap_{p\geq 1}L^p(\mathbb{R}^{3})$ with support in $\Omega$ such that
 \begin{equation}\label{B}
  \lim_{a\rightarrow 0}U^\infty(\hat{x},\theta)= U_{0}^\infty(\hat{x},\theta) \mbox{ uniformly in terms of } \theta \mbox{ and } \hat{x} \mbox{ in } \mathbb{S}^2
 \end{equation}
where $U_{0}^\infty(\hat{x},\theta)$ is the farfield corresponding to the scattering problem

\begin{equation}
(\Delta + \kappa^{2}-(K+1)\bold{C}_0)U_{0}^{t}=0 \mbox{ in }\mathbb{R}^{3},\label{B-1}
\end{equation}
\begin{equation}
U_{0}^{t}=U_{0}^s +e^{i\kappa x\cdot \theta},  
\end{equation}
\begin{equation}
\frac{\partial U_{0}^{s}}{\partial |x|}-i\kappa U_{0}^{s}=o\left(\frac{1}{|x|}\right), |x|\rightarrow\infty. \label{radiationc-B-1}
\end{equation}
 
\item If in addition $K\mid_{\Omega}$ is in $C^{0, \gamma}(\Omega)$, $\gamma \in (0, 1]$ and the obstacles have the same capacitances \footnote{The same result holds if we consider 
the obstacles to
have locally the same capacitances, as in point 1.}, then
\begin{equation}\label{C}
  U^\infty(\hat{x},\theta)= U_0^\infty(\hat{x},\theta) +O(a^{\min\{\gamma, \frac{1}{3}-\frac{4}{5}t\}}) \mbox{ uniformly in terms of } \theta \mbox{ and } \hat{x} \mbox{ in } \mathbb{S}^2
 \end{equation}
 where $C_0=C$ in $\Omega$ and $C_0=0$ in $\mathbb{R}^{3} \setminus{\overline \Omega}$.
\end{enumerate}

\end{theorem}
\bigskip

The interesting observation behind such results is the 'equivalent' behaviour between a collection of, appropriately dense, small holes (or impenetrable obstacles)
and an extended penetrable obstacle modeled by an additive potential. Such an observation goes back at least to the works by Cioranescu and Murat \cite{C-M:1979, C-M:1997} and also
the reference therein. Their analysis, made for the Poisson problem, is based on homogenization via energy methods and, in particular, they assume that the obstacles are distributed periodically. 
More elaborated expositions on the homogenization theory applied to related problems can be found in the books \cite{B-L-P:1978} and \cite{J-K-O:1994}. 
\bigskip

In the results presented here, we do not need such periodicity and no homogenization is used. Instead, we first use integral equation methods to derive the asymptotic expansion 
(\ref{x oustdie1 D_m farmain-recent**}), which is deduced from \cite{C-S:2014}, and second we analyze the limit, as $a$ goes to zero, of the dominant term 
$\sum_{m=1}^{M}e^{-i\kappa\hat{x}\cdot z_m}Q_m$ when $M:=M(a):=O(a^{-1})$ recalling that $Q_m, m=1, ..., M$ solves the Foldy-Lax algebraic system (\ref{fracqcfracmain}). 
The main ingredients in this analysis are related to the invertibility properties of this last algebraic system derived in \cite{C-S:2014}, see subsection \ref{FL-system} below, 
and the precise treatment of the summation in the formentioned dominant term. As we can see in Theorem \ref{equivalent-medimu}, the equivalent term (or the strange term recalling
the terminology of Cioranescu and Murat) is composed of two terms. The first one, $K+1$, models the local number of the distributed obstacles while the second one, $\bold C_0$, models
their geometry. In the situation discussed in \cite{C-M:1979, C-M:1997}, and other references, the coefficient $K$ is reduced to zero since locally they have only one obstacle 
and then $t=\frac{1}{3}$, see Figure \ref{distribution-obstacles}. It happens that this coefficient can have interesting applications in the theory of acoustic materials.
Indeed, perforating a given domain by a set of holes 
having the same shape (balls for instance) and distributed in an appropriate way following a given function $K$, see the paragraph before Remark \ref{distribution}, 
then the asymptotic expansion in (\ref{B}) says that
the farfield generated by such a collection is equivalent to the one corresponding to an acoustic medium having $n(x):=\sqrt{1-\frac{(K+1)\bold C_0}{\kappa}}$ 
as an index of refraction. We can also use the geometry, i.e. the coefficient $\bold{C}_0$, instead of $K$ to derive the same conclusion. 
The error estimates in (\ref{C}) measures the error between the scattered fields generated by the perforated medium and the ones related to the refraction index.
In other words, these estimates measure the accuracy in the design, by perforation, of acoustic materials with desired refraction index.

\bigskip

Let us make some additional comments on these results related to the inversion theory. 
Let the small scatterers model small anomalies (i.e. tumors). Saying that the collection of the scatterers is dense ($M$ is large, $a$ and $d$ are small), means that the tumor propagates 
and becomes an advanced one. In this case, the equivalent medium\footnote{ We choose the terminology 'equivalent medium' instead of 'effective medium'. 
The reason is that we only estimate the farfields (or the scattered fields away from the location of the obstacles). In effective medium theory, usually we derive the limit of the energy everywhere.
}  is what we could see from the measurements collected far away.
If we have access to the measured farfields corresponding to the distributed obstacles described above, then the equivalent 
medium can be described and quantified by solving the inverse potential scattering problem $\{U_{0}^\infty(\hat{x},\theta), \; \theta \mbox{ and } \hat{x} \mbox{ in } \mathbb{S}^2\} 
\rightarrow (K+1)\bold{C}_0$.  
In this case the error estimate in (\ref{C}) added to the stability estimate of the inverse scattering problem help to reconstruct the effective medium from these measurements.
This inverse problem is quite well studied using the methods introduced in \cite{Nachman:1988, Novikov:1988, Ramm:1988} for instance.

\bigskip

A result similar to (\ref{B}) is also derived by Ramm in several of his papers, see for instance \cite{RAMM:2007}, where in addition to some formal arguments, he needs some extra 
assumptions on the distribution of the obstacles to ensure the validity of some integral formulas. The additional contribution of our work compared to his results is
that we provide asymptotics expansions with explicit error estimates, as in (\ref{C}).

\bigskip

So far we studied the case when $M$ is of the order $a^{-s},\; s \in [0,1]$, as $a\rightarrow 0$. We finish this introduction by claiming that, in the case when it is of the order
$a^{-s},\; s>1$, the equivalent medium is the exterior impenetrable and soft obstacle $\Omega$. However, we think that its justification and the corresponding error estimates 
are out reach by the mathematical tools we use in this paper. In a forthcoming work, we will analyze this situation and quantify the corresponding error estimates. 
 
\bigskip

The rest of the paper is devoted to the proof of Theorem \ref{equivalent-medimu}. We proceed as follows. 
In subsection \ref{FL-system}, we recall the  invertibility of the algebraic system \eqref{fracqcfracmain} derived from \cite{C-S:2014}. 
Then, in subsection \ref{subsection-piecewise-constant}, we deal with the case when the coefficient $K$ is piecewise constant and the osbtacles are locally the same, 
i.e. $\bold{C}$ is piecewise constant, by dividing $\Omega$ into $N$ regions $E_j$, $j=1, ..., N$. In section \ref{arbitrarely-distributed}, we apply the results of section \ref{subsection-piecewise-constant}
to the case when $N=[a^{-1}]$ and $E_j=\Omega_j$, $j=1,..., N$, and then pass to the limit $a\rightarrow 0$. In section \ref{smoothly-distributed}, we deal as in 
section \ref{subsection-piecewise-constant} using the H$\ddot{\mbox{o}}$lder regularity of $K$.

\section{Proof of the results}\label{main-section}
\subsection{Invertibility properties of the Foldy-Lax algebraic system}\label{FL-system}
 We can rewrite the algebraic system \eqref{fracqcfracmain} as follows;

\begin{eqnarray}\label{fracqcfracmain-effect}
 Y_m +\sum_{\substack{j=1 \\ j\neq m}}^{M} \Phi_\kappa(z_m,z_j)\bar{C}_j Y_j a&=&-U^{i}(z_m, \theta),
\end{eqnarray}
with $Y_m:=\frac{Q_m}{C_m}$ and $C_m:=\bar{C}_m a$, for $m=1,\dots,M$ and $\bar{C}_m$ are the capacitances of $B_m$'s, i.e. they are independent of $a$. We set $\hat{C}:=(\bar{C}_1,\bar{C}_2,\dots,\bar{C}_M)^\top$ 
and define
\begin{eqnarray*}
\mathbf{B}\hspace{-.2cm}:=\hspace{-.2cm}\left(\begin{array}{ccccc}
   -1 &-\bar{C}_2a\Phi_\kappa(z_1,z_2)&-\bar{C}_3a\Phi_\kappa(z_1,z_3)&\cdots&-\bar{C}_Ma\Phi_\kappa(z_1,z_M)\\
-\bar{C}_1a\Phi_\kappa(z_2,z_1)&-1&-\bar{C}_3a\Phi_\kappa(z_2,z_3)&\cdots&-\bar{C}_Ma\Phi_\kappa(z_2,z_M)\\
 \cdots&\cdots&\cdots&\cdots&\cdots\\
-\bar{C}_1a\Phi_\kappa(z_M,z_1)&-\bar{C}_2a\Phi_\kappa(z_M,z_2)&\cdots&-\bar{C}_{M-1}a\Phi_\kappa(z_M,z_{M-1}) &-1
   \end{array}\right),\label{mainmatrix-acoustic-small}\\
\nonumber\\
 \hspace{-4cm}{\hat{Y}}:=\left(\begin{array}{cccc}
    Y_1 & Y_2 & \ldots  & Y_M
   \end{array}\right)^\top \text{ and } 
\mathrm{U}^I:=\left(\begin{array}{cccc}
     U^i(z_1) & U^i(z_2)& \ldots &  U^i(z_M)
   \end{array}\right)^\top.\hspace{2cm}\label{coefficient-and-incidentvectors-acoustic-small}
\end{eqnarray*}

The following lemma insures the invertibility of the algebraic system (\ref{fracqcfracmain-effect}).

\begin{lemma}\label{Mazyawrkthm-effect}
If $a<\frac{5\pi}{3}\frac{d}{\|\hat{C}\|}$ and $t:=\min\limits_{j\neq\,m,1\leq\,j,m\leq\,M}\cos(\kappa|z_m-z_j|) \geq 0$, then the matrix $\mathbf{B}$ 
is invertible and the solution vector $\hat{Y}$ of \eqref{fracqcfracmain-effect} satisfies the estimate
\begin{equation}\label{mazya-fnlinvert-small-ac-2-effect}
 \sum_{m=1}^{M}|Y_m|^{2}
\leq4\left(1-\frac{3ta}{5\pi\,d}\|\hat{C}\|\right)^{-2}\sum_{m=1}^{M}\left|U^i(z_m)\right|^2,
\end{equation}
and hence the estimate
\begin{equation}\label{mazya-fnlinvert-small-ac-3-effect}
\begin{split}
 \sum_{m=1}^{M}|Y_m|
\leq2\left(1-\frac{3ta}{5\pi\,d}\|\hat{C}\|\right)^{-1}M\max\limits_{1\leq m \leq M}\left|U^i(z_m)\right|.
\end{split}
\end{equation}
\end{lemma}

The proof of this lemma can be found in \cite{C-S:2014}.

\subsection{Case when the obstacles are locally the same}\label{subsection-piecewise-constant} 
 
We define a bounded function $K^M: \mathbb{R}^3\rightarrow \mathbb{R}$ as follows: 
\begin{equation}
K^M(x):= K^M(z_m):=
\left\{\begin{array}{ccc}
K(z_m)+1&   \mbox{ if }& x\in \Omega_m \\
0 & \mbox{ if }& x\notin \Omega_m \mbox{ for any } m=1,\dots,[a^{-1}].
\end{array}\right.
\end{equation}
Hence each $\Omega_m$ contains $[ K^M(z_m)]$ obstacles and $K_{max}:=\sup_{z_m}K^M(z_m)$.
 
Let  ${C^{M}}$ be a piecewise constant function such that ${C^{M}}\vert_{\Omega_m}=\bar{C}_m$ for all $m=1,\dots,M$ and vanishes outside $\Omega$. 
We assume that there exists a finite family of $E_j, j=, 1, ..., N$ such that $\Omega=\cup^N_{j=1} E_j$ and that each $E_j$ contains the same obstacles. 
Hence it is clear that ${C^{M}}$ is piecewise constant, i.e. 
 ${C^{M}}\arrowvert_{E_j}=\bar{C}^j\,\big(\in\{\bar{C}_m\}_{m=1}^{M}\big)$, a constant. 
Let, for each $j=1,\dots, N$, $E_j$ contains $N_j$ cubes $\{\Omega^j_l\}_{l=1}^{N_j}\,\big(\subset\{\Omega_m\}_{m=1}^{M}\big)$, 
which contains the obstacles $\{D^j_l\}_{l=1}^{N_j}\,\big(\subset\{D_m\}_{m=1}^{M}\big)$,
associated to the same reference body, i.e. ,
\begin{eqnarray}
 &\sum_{j=1}^{N}N_j=M,&\label{def1verxi}\\
& B^j_{l1}=B^j_{l2} \mbox{ for } l1,l2\in\{1,\dots, N_j\}& \mbox{ with }  \{B^j_l\}_{l=1}^{N_j}\subset\{B_m\}_{m=1}^{M}.\label{def2verxi}\\
 \mbox{and }& \{\bar{C}^j\}_{j=1}^{N}=\{\bar{C}^M\}_{m=1}^{M}={}_{j=1}^{N}\{\bar{C}^{j}_{l_j}\}_{l_j=1}^{N_j}.
\end{eqnarray}

We set $\mathbf{C}:=\max\limits_{1\leq{m}\leq{M}}\bar{C}_m$, then we have 
\begin{eqnarray}
 \mathbf{C}&=&\max\limits_{1\leq{j}\leq{N}}\max\limits_{1\leq {l} \leq {N_j} }\bar{C}^{j}_{l}\nonumber\\
 &=&\max\limits_{1\leq{j}\leq{N}}\bar{C}^{j}_{l_j}\quad \mbox{for each}\quad l_j\in\{1,\dots, N_j\}\nonumber\\
  &=&\max\limits_{1\leq{j}\leq{N}}\bar{C}^{j}.
\end{eqnarray}


\bigskip\par Consider the Lippmann-Schwinger equation
\begin{eqnarray}\label{fracqcfracmain-effect-int}
 Y(z) +\int_{\Omega} \Phi_\kappa(z,y)K^M(y){C^{M}}(y) Y(y) dy &=&-U^{i}(z, \theta), z\in \Omega
\end{eqnarray}
and define
\begin{eqnarray}
 V(Y)(x):=\int_{\Omega}\Phi_\kappa(x,y)K^M(y){C^{M}}(y)Y(y)dy,\qquad x\in\mathbb{R}^3,
\end{eqnarray}
then we can show that $V:{L}^2(\Omega)\rightarrow {H}^2(\varOmega)$ is a bounded operator for any bounded domain $\varOmega$ in $\mathbb{R}^3$, see \cite{C-K:1998}, and  in particular there exists a positive 
constant $c_0$ such that
\begin{eqnarray}\label{H2normofY}
\|V(Y)\|_{H^{2}(\Omega)}\leq c_0 \|Y\|_{L^{2}(\Omega)}.
\end{eqnarray} 
 We have also the following lemma
 \bigskip
 
\begin{lemma}\label{invertibility-of-VC}
 There exists one and only one solution $Y$ of the Lippmann-Schwinger equation (\ref{fracqcfracmain-effect-int}) and it satisfies the estimate
 \begin{eqnarray}\label{est-Lipm-Sch}
  \Vert Y\Vert_{L^\infty(\Omega)}\leq C \Vert U^i\Vert_{H^2(\Omega)}& \mbox{ and }& \Vert \nabla Y\Vert_{L^\infty(\Omega)}\leq C^{\prime} \Vert U^i\Vert_{H^2(\tilde{\Omega})},
 \end{eqnarray}
where $\tilde{\Omega}$ being a large bounded domain which contains $\bar{\Omega}$.
\end{lemma}

\begin{proof}{\it{of Lemma \ref{invertibility-of-VC}}}

The proof of the existence and uniqueness is guaranteed by the Fredholm alternative applied to $I +V:L^2(\Omega) \longmapsto L^2(\Omega)$, see \cite{C-K:1998} for instance. 
Let us derive the estimates in (\ref{est-Lipm-Sch}). From the invertibility of the equation (\ref{fracqcfracmain-effect-int}) from $L^2(\Omega)$ to $L^2(\Omega)$, we deduce that 
$\Vert Y\Vert_{L^2(\Omega)}\leq c_1 \Vert U^i\Vert_{L^2(\Omega)}$. In addition 
\begin{eqnarray}
\Vert Y\Vert_{H^2(\Omega)} & \leq & \Vert V(Y)\Vert_{H^2(\Omega)}+ \Vert U^i\Vert_{H^2(\Omega)}\nonumber\\
& \leq &  c_0 \Vert Y\Vert_{L^2(\Omega)}+ \Vert U^i\Vert_{H^2(\Omega)}\; (\mbox{ using } (\ref{H2normofY}))\nonumber\\
& \leq & C \Vert U^i\Vert_{H^2(\Omega)} (\mbox{ by the invertibility of } (\ref{fracqcfracmain-effect-int}) \mbox{ in }  L^2(\Omega)).
\end{eqnarray}
The proof of the first part ends by the Sobolev embedding $H^2(\Omega) \subset L^\infty(\Omega)$.
\bigskip

Let us prove the second part. From \eqref{fracqcfracmain-effect-int}, it can be shown that $\bar{Y}$ satisfies the partial differential equation;
\begin{eqnarray}\label{poissonY}
 \Delta \bar{Y} &=& -(\kappa^2-K^M{C^{M}}) \bar{Y} \quad \mbox{in}\quad \mathbb{R}^3  
\end{eqnarray}
where, with an abuse of notation, 
\begin{eqnarray}
 \bar{Y}(x):=\left\{\begin{array}{ccc}
                  Y(x)      &\quad    \mbox{in}&\quad \Omega \\
                  -\int_{\Omega}\Phi_\kappa(x,y)K^M(y){C^{M}}(y)Y(y)dy-U^i  &\quad    \mbox{in}&\quad \mathbb{R}^3\setminus\overline{\Omega}. \\
                 \end{array}
                 \right.
\end{eqnarray}

Let $\tilde{\Omega}$ be a large bounded domain which contains $\bar{\Omega}$, then, by the interior estimates we deduce from (\ref{poissonY}) that 
 there exist a constant $c_3$ such that
\begin{eqnarray}\label{interY}
 \|Y\|_{W^{2,p}(\Omega)} \leq&c_4\|Y\|_{L^p(\tilde{\Omega})}
\end{eqnarray}
where $c_4:=(\kappa^2+K_{max}\mathbf{C})$. Again from the boundedness of the operator $V$, one can obtain $\|Y\|_{L^p(\tilde{\Omega})}\leq c_p\|Y\|_{H^2(\tilde{\Omega})}\leq c_pc_0\|Y\|_{L^2({\Omega})} +
c_p\Vert U^i\Vert_{H^2(\tilde\Omega)}$ for $p>0$. It allows us to write \eqref{interY} as

\begin{eqnarray}\label{interY-1}
  \|Y\|_{W^{2,p}(\Omega)}&\leq&c_4c_pc_0\|Y\|_{L^2({\Omega})} + c_4c_p\Vert U^i\Vert_{H^2(\tilde\Omega)}\nonumber\\
                        &\leq&c_4c_pc_0|\Omega|^{\frac{1}{2}}\|Y\|_{L^\infty({\Omega})}+c_4c_p\Vert U^i\Vert_{H^2(\tilde\Omega)} \nonumber\\
                        &\leq& c_5 C \Vert U^i\Vert_{H^2(\tilde\Omega)}                      
\end{eqnarray}
 where $c_5:=c_4c_p\left(c_0|\Omega|^{\frac{1}{2}}+\frac{1}{C}\right)$. The Sobolev embeding $W^{1, p}(\Omega) \subset L^\infty(\Omega)$, for $p>3$, implies that
\begin{eqnarray}\label{interY-2}
  \|\nabla Y\|_{L^{\infty}(\Omega)}\leq\|\nabla Y\|_{W^{1,p}(\Omega)}\leq\| Y\|_{W^{2,p}(\Omega)}
\end{eqnarray}

for $p>3$. Hence, \eqref{interY-1} and \eqref{interY-2} give us the estimate

\begin{eqnarray}\label{interY-3}
  \|\nabla Y\|_{L^{\infty}(\Omega)}\leq C^{\prime}\Vert U^i\Vert_{H^2(\tilde\Omega)}.
\end{eqnarray}

\end{proof}

 Write $c_1:=C \Vert U^i\Vert_{H^2(\tilde\Omega)}$, then we rewrite the estimates given in \eqref{est-Lipm-Sch} as below
  \begin{eqnarray}\label{est-Lipm-Sch-1}
  \Vert Y\Vert_{L^\infty(\Omega)}\leq c_1& \mbox{ and }& \Vert \nabla Y\Vert_{L^\infty(\Omega)}\leq c_5c_1.
 \end{eqnarray}

Observe that, for $m=1,\dots,M$, equation \eqref{fracqcfracmain-effect-int} can be rewritten as
\begin{eqnarray}\label{fracqcfracmain-effect-int-1}
 Y(z_m) +\sum_{\substack{j=1 \\ j\neq m}}^{M} \Phi_\kappa(z_m,z_j)\bar{C}_j Y(z_j) a&=&-U^{i}(z_m, \theta)\\
 && +\sum_{\substack{j=1 \\ j\neq m}}^{M} \Phi_\kappa(z_m,z_j)\bar{C}_j Y(z_j) a-\sum_{\substack{j=1 \\ j\neq m}}^{[a^{-1}]} \Phi_\kappa(z_m,z_j)K^M(z_j)\bar{C}_j Y(z_j) Vol(\Omega_j)\nonumber\\
 &&+\sum_{\substack{j=1 \\ j\neq m}}^{[a^{-1}]} \Phi_\kappa(z_m,z_j)K^M(z_j)\bar{C}_j Y(z_j) Vol(\Omega_j)-\int_{\Omega} \Phi_\kappa(z_m,y)K^M(y){C^{M}}(y) Y(y) dy.\nonumber
\end{eqnarray}

We set 
$$
A:=\sum_{\substack{j=1 \\ j\neq m}}^{[a^{-1}]} \Phi_\kappa(z_m,z_j)K^M(z_j)\bar{C}_j Y(z_j) Vol(\Omega_j)-\int_{\Omega} \Phi_\kappa(z_m,y)K^M(y){C^{M}}(y) Y(y) dy
$$
and 
$$
B:= \sum_{\substack{j=1 \\ j\neq m}}^{M} \Phi_\kappa(z_m,z_j)\bar{C}_j Y(z_j) a-\sum_{\substack{j=1 \\ j\neq m}}^{[a^{-1}]} \Phi_\kappa(z_m,z_j)K^M(z_j)\bar{C}_j Y(z_j) Vol(\Omega_j).
$$

\subsubsection{Estimate of $A$}
In order to evaluate this, first observe that
\begin{itemize}
 \item \begin{equation}\label{integralonomega}
            \int_{\Omega} \Phi_\kappa(z_m,y){K^{M}}(y){C^{M}}(y) Y(y) dy=\sum_{l=1}^{[a^{-1}]}\int_{\Omega_l} \Phi_\kappa(z_m,y){K^{M}}(y){C^{M}}(y) Y(y) dy.
       \end{equation}
     
 \item For $l\neq m$, we have
 \begin{eqnarray}\label{integralonomega-subelements}
  \int_{\Omega_l} \Phi_\kappa(z_m,y){K^{M}}(y){C^{M}}(y) Y(y) dy - \Phi_\kappa(z_m,z_l){K^{M}}(z_l)\bar{C}_l Y(z_l) Vol(\Omega_l) \\
  &\hspace{-3cm}=\quad{K^{M}}(z_l)\bar{C}_l\int_{\Omega_l} \left[\Phi_\kappa(z_m,y) Y(y) - \Phi_\kappa(z_m,z_l) Y(z_l)\right] dy.\nonumber
 \end{eqnarray}
  Write, $f(z_m,y)=\Phi_\kappa(z_m,y) Y(y)$. Using Taylor series, we can write
  $$f(z_m,y)-f(z_m,z_l)=(y-z_l)R_l(z_m,y),$$
  with 
  \begin{eqnarray}\label{taylorremind1}
   R_l(z_m,y)
   &=&\int_0^1\nabla_y f(z_m,y-\beta(y-z_l))\,d\beta\nonumber\\
   &=&\int_0^1\nabla_y \left[\Phi_\kappa(z_m,y-\beta(y-z_l)) Y(y-\beta(y-z_l))\right]\,d\beta\nonumber\\
   &=&\int_0^1\left[\nabla_y\Phi_\kappa(z_m,y-\beta(y-z_l))\right] Y(y-\beta(y-z_l))\,d\beta\nonumber\\
   &&+\int_0^1\Phi_\kappa(z_m,y-\beta(y-z_l))\left[\nabla_y Y(y-\beta(y-z_l))\right]\,d\beta.
  \end{eqnarray}
For $m=1,\dots,[a^{-1}]$ fixed, we distinguish between the cubes $\Omega_j$, $j\neq\,m$ which are near to $\Omega_m$ from the ones which are far from $\Omega_m$ as follows.  Let $\Omega^\prime_m$, $1\leq\,m\leq\,a^{-1}$ be the balls
of center $z_m$ and of radius $(\frac{1}{2}a^\frac{\alpha}{3})$ with $0<\alpha\leq1$. 
The cubes lying in $\Omega^\prime_m$ will fall into the category, $N_{\Omega_m}$, of near by cubes
 and the others into the category, $F_{\Omega_m}$, of far cubes to $\Omega_m$. Since  $\Omega^\prime_m$ are balls with same diameter and $\Omega_m$ are cubes of 
  volume $a\frac{[K^M(z_m)]}{K^M(z_m)}$, the number of cubes
 near by $\Omega_m$ will not exceed $\frac{\Pi}{3}a^{\alpha-1}$   
 $\left[:=\frac{\frac{4}{3}\pi\left((\frac{1}{2}a^\frac{\alpha}{3})\right)^3}{a\min\limits_m\frac{[K^M(z_m)]}{K^M(z_m)}}\right]$ observing that  $\frac{1}{2}\leq\frac{[K^M(z_m)]}{K^M(z_m)}\leq1$.
 \par To make sure that at least $\Omega_m\subset\Omega^\prime_m$ for each $m$, we need to have ${{a}^\frac{\alpha}{3}\geq\sqrt{2}{a}^\frac{1}{3}}$. i.e. $\alpha\geq2^{-\frac{3}{2}}$. 
 \par Also, we have that for the cubes $\Omega_j\in F_{\Omega_m}$, we have $dist(z_m, y)\geq \frac{1}{2}{a}^\frac{\alpha}{3}$ for all $y\in\Omega_j$ and for the cubes $\Omega_{j'}\in N_{\Omega_m}\setminus\{\Omega_m\}$, we have
 $dist(z_m, y)\geq \frac{\left(a\frac{[K^M(z_m)]}{K^M(z_m)}\right)^{\frac{1}{3}}}{2}\geq\frac{\left(\frac{a}{2}\right)^{\frac{1}{3}}}{2}\geq\frac{d}{2^\frac{4}{3}}$ for all $y\in\Omega_{j'}$. 

  From the explicit form of $\Phi_\kappa$, we have $\nabla_y\Phi_\kappa(x,y)=\Phi_\kappa(x,y)\left[\frac{1}{|x-y|}-i\kappa\right]\frac{x-y}{|x-y|}, {x}\neq{y}$.   
  Hence, we obtain that
  \begin{itemize}
   \item for $l\neq m$ such that $\Omega_l\in F_{\Omega_m}$, we have
   \begin{eqnarray*}
   \vert\Phi_\kappa(z_m,y-\beta(y-z_l))\vert\leq\frac{1}{2\pi{{a}^\frac{\alpha}{3}}},& \mbox{ and }&\vert\nabla_y\Phi_\kappa(z_m,y-\beta(y-z_l))\vert\leq\frac{1}{2\pi{{a}^\frac{\alpha}{3}}}\left[\frac{2}{{a}^\frac{\alpha}{3}}+\kappa\right].
  \end{eqnarray*}
 These values give us
 \begin{eqnarray}\label{taylorremind-effect-F}
 \vert R_l(z_m,y) \vert
   &\leq&\frac{1}{2\pi{{a}^\frac{\alpha}{3}}}\left(\left[\frac{2}{{a}^\frac{\alpha}{3}}+\kappa\right]\int_0^1 {\vert Y(y-\beta(y-z_l))\vert}d\beta+\int_0^1{\vert\nabla_y Y(y-\beta(y-z_l))\vert}d\beta\right).\nonumber\\
  \end{eqnarray}
  \item for $l\neq m$ such that $\Omega_l\in N_{\Omega_m}$, we have
   \begin{eqnarray*}
   \vert\Phi_\kappa(z_m,y-\beta(y-z_l))\vert\leq\frac{2^\frac{1}{3}}{2\pi{d}},& \mbox{ and }&\vert\nabla_y\Phi_\kappa(z_m,y-\beta(y-z_l))\vert\leq\frac{2^\frac{1}{3}}{2\pi{d}}\left[\frac{2^\frac{4}{3}}{d}+\kappa\right].
  \end{eqnarray*}
 These values give us
 \begin{eqnarray}\label{taylorremind-effect}
 \hspace{-.5cm}\vert R_l(z_m,y) \vert
   &\leq&\frac{2^\frac{1}{3}}{2\pi{d}}\left(\left[\frac{2^\frac{4}{3}}{d}+\kappa\right]\int_0^1 {\vert Y(y-\beta(y-z_l))\vert}d\beta+\int_0^1{\vert\nabla_y Y(y-\beta(y-z_l))\vert}d\beta\right).
  \end{eqnarray}
  \end{itemize}
  
  Then, for $l\neq m$ such that $\Omega_l\in N_{\Omega_m}$, \eqref{integralonomega-subelements} and \eqref{taylorremind-effect} imply the estimate
  \begin{eqnarray}\label{integralonomega-subelements-abs}
  \left\vert\int_{\Omega_l} \Phi_\kappa(z_m,y){K^{M}}(y){C^{M}}(y) Y(y) dy - \Phi_\kappa(z_m,z_l){K^{M}}(z_l)\bar{C}_l Y(z_l) Vol(\Omega_l)\right\vert \nonumber \\
  &\hspace{-13cm}\leq\quad\frac{2^\frac{1}{3}{K^{M}}(z_l)\bar{C}_l}{2\pi{d}}\left(\left[\frac{2^\frac{4}{3}}{d}+\kappa\right]\int_{\Omega_l} \left[\int_0^1 {\vert Y(y-\beta(y-z_l))\vert}d\beta\right]\vert y-z_l\vert dy\right.\nonumber\\
  &\hspace{-11cm}\left.+\int_{\Omega_l} \left[\int_0^1 {\vert \nabla_yY(y-\beta(y-z_l))\vert}d\beta\right]\vert y-z_l\vert dy\right)\nonumber\\
  &\hspace{-14.5cm}\substack{\leq\\Lemma \ref{invertibility-of-VC}}\quad\frac{2^\frac{1}{3}{K^{M}}(z_l)\bar{C}_l}{2\pi{d}}\,a^\frac{1}{3}\,\left(\left[\frac{2^\frac{4}{3}}{d}+\kappa\right]c_1\int_{\Omega_l}  dy\right.
  \left.+  |\Omega_l|\|\nabla {Y}\|_{L^\infty(\Omega_l)} \right)\nonumber\\
    &\hspace{-17cm}\substack{\leq\\ \eqref{interY-3}}\quad c_12^\frac{1}{3}\frac{\textcolor{black}{[{K^{M}}(z_l)]}\bar{C}_l}{2\pi{d}}\,a\,\,a^\frac{1}{3}\,\left(\left[\frac{2^\frac{4}{3}}{d}+\kappa\right]\right.
  \left.+  c_5 \right).
 \end{eqnarray}
In the similar way, for $l\neq m$ such that $\Omega_l\in F_{\Omega_m}$, we obtain the following estimate using \eqref{integralonomega-subelements} and \eqref{taylorremind-effect-F} ;
\begin{eqnarray}\label{integralonomega-subelements-abs-F}
  \left\vert\int_{\Omega_l} \Phi_\kappa(z_m,y){K^{M}}(y){C^{M}}(y) Y(y) dy - \Phi_\kappa(z_m,z_l){K^{M}}(z_l)\bar{C}_l Y(z_l) Vol(\Omega_l)\right\vert \nonumber\\
   &\hspace{-6cm}\leq\quad c_1\frac{\textcolor{black}{[{K^{M}}(z_l)]}\bar{C}_l}{2\pi{{a}^\frac{\alpha}{3}}}\left(\left[\frac{2}{{a}^\frac{\alpha}{3}}+\kappa\right]+c_5\right)\,a^\frac{1}{3}\,a. 
 \end{eqnarray}
 \item Let us estimate the integral value $\int_{\Omega_m} \Phi_\kappa(z_m,y){C^{M}}(y) Y(y) dy$. We have the following estimates:
\begin{eqnarray}\label{estmatemthint-effe-acc}
 \left\vert\int_{\Omega_m} \Phi_\kappa(z_m,y){K^{M}}(y){C^{M}}(y) Y(y) dy\right\vert
 &\leq&c_1{K^{M}}(z_m)\bar{C}_m\left\vert\int_{\Omega_m} \Phi_\kappa(z_m,y) dy\right\vert\nonumber\\
 &\leq&\frac{1}{4\pi}c_1{K^{M}}(z_m)\bar{C}_m\left(\int_{B(z_m,r)} \frac{1}{|z_m-y|} dy+\int_{\Omega_m\setminus B(z_m,r)} \frac{1}{|z_m-y|} dy\right)\nonumber\\
 &&{(\frac{1}{|z_m-y|} \in L^1(B(z_m,r)), r<\frac{a^{\frac{1}{3}}}{2} )}\nonumber\\
 &\leq&\frac{1}{4\pi}c_1{K^{M}}(z_m)\bar{C}_m\left(\sigma(\mathbb{S}^{3-1})\int_0^r \frac{1}{s}s^{3-1}ds+\frac{1}{r}Vol(\Omega_m\setminus B(z_m,r))\right)\nonumber\\
 &=&\frac{1}{4\pi}c_1{K^{M}}(z_m)\bar{C}_m\underbrace{\left(2\pi r^2 +\frac{1}{r}\left[a-\frac{4}{3}\pi r^3\right]\right)}_{=:lm(r,a)}\nonumber\\
 &\leq&\frac{1}{4\pi}c_1{K^{M}}(z_m)\mathbf{C}~lm(r^c,a),\nonumber\\
 &&\mbox{ $r^c$ is the value of $r$ where $lm(r,a)$ attains maximum}.\nonumber\\
 &&{\partial_r lm(r,a)=0 \Rightarrow 4\pi r-\frac{a}{r^2}-\frac{8}{3}\pi r=0\Rightarrow  r_c=\left(\frac{3}{4}\pi a\right)^\frac{1}{3}}\nonumber\\
 &&{\begin{array}{ccc}
                    lm(r_c,a)&=&2\pi\left(\frac{3}{4}\pi\right)^\frac{2}{3} a^\frac{2}{3}+\left(\frac{4}{3\pi} \right)^\frac{1}{3}a^\frac{2}{3}-\frac{4}{3}\pi\left(\frac{3}{4}\pi \right)^\frac{2}{3}a^\frac{2}{3}\\
                    &&\\
                    &=&\left[\frac{2}{3\pi}\left(\frac{3}{4}\pi\right)^\frac{2}{3}+\left(\frac{4}{3\pi} \right)^\frac{1}{3} \right]a^\frac{2}{3}=\frac{3}{2}\left(\frac{4}{3\pi} \right)^\frac{1}{3}a^\frac{2}{3}
                   \end{array}
}\nonumber\\ 
&=&\frac{3}{8\pi}c_1K_{max}\mathbf{C}\left(\frac{4}{3\pi} \right)^\frac{1}{3}a^\frac{2}{3}.
\end{eqnarray}
\end{itemize}

\subsubsection{Estimate of $B$}
$$\sum_{\substack{j=1 \\ j\neq m}}^{M} \Phi_\kappa(z_m,z_j)\bar{C}_j Y(z_j)a -\sum_{\substack{j=1 \\ j\neq m}}^{[a^{-1}]} \Phi_\kappa(z_m,z_j)K^M(z_j)\bar{C}_j Y(z_j) Vol(\Omega_j)=$$
$$
 \sum_{\substack{l=1 \\ l\neq m\\ z_l \in \Omega_m}}^{\textcolor{black}{[K^M(z_m)]}}\Phi_\kappa(z_m,z_l)\bar{C}_l Y(z_l) a+
\sum_{\substack{j=1 \\ j\neq m}}^{[a^{-1}]} \sum_{\substack{l=1 \\ z_l \in \Omega_j}}^{\textcolor{black}{[K^M(z_j)]}}\Phi_\kappa(z_m,z_l)\bar{C}_l Y(z_l) a-
\sum_{\substack{j=1 \\ j\neq m}}^{[a^{-1}]} \Phi_\kappa(z_m,z_j)K^M(z_j)\bar{C}_j Y(z_j) Vol(\Omega_j)=
$$
$$
\bar{C}_m a\sum_{\substack{l=1 \\ l\neq m\\ z_l \in \Omega_m}}^{\textcolor{black}{[K^M(z_m)]}}\Phi_\kappa(z_m,z_l) Y(z_l) +\sum_{\substack{j=1 \\ j\neq m}}^{[a^{-1}]}\bar{C}_j a\big[\big(\sum_{\substack{l=1 \\ z_l \in \Omega_j}}^{\textcolor{black}{[K^M(z_j)]}}\Phi_\kappa(z_m,z_l) Y(z_l)\big)-
 \Phi_\kappa(z_m,z_j)\textcolor{black}{[K^M(z_j)]} Y(z_j)\big],$$
 since $ Vol(\Omega_j)=a\frac{\textcolor{black}{[K^M(z_j)]}}{K^M(z_j)}\; \mbox{ and } \bar{C}_l=\bar{C}_j, \mbox{ for } l=1, ...,\; K^M(z_j).
$
We write,

\begin{eqnarray}\label{Ej1}
E^j_1&:=&\sum_{\substack{l=1 \\ l\neq m\\ z_l \in \Omega_m}}^{\textcolor{black}{[K^M(z_m)]}}\Phi_\kappa(z_m,z_l) Y(z_l)
\end{eqnarray}
 and
\begin{eqnarray}\label{Ej2}
E^j_2&:=&\big[\big(\sum_{\substack{l=1 \\ z_l \in \Omega_j}}^{\textcolor{black}{[K^M(z_j)]}}\Phi_\kappa(z_m,z_l) Y(z_l)\big)-
 \Phi_\kappa(z_m,z_j)\textcolor{black}{[K^M(z_j)]} Y(z_j)\big]\nonumber\\
 &=&\sum_{\substack{l=1 \\ z_l \in \Omega_j}}^{\textcolor{black}{[K^M(z_j)]}}\big(\Phi_\kappa(z_m,z_l) Y(z_l)-
 \Phi_\kappa(z_m,z_j) Y(z_j)\big).
\end{eqnarray}

We need to estimate $\bar{C}_m a E^j_1$ and $\sum_{\substack{j=1 \\ j\neq m}}^{[a^{-1}]}\bar{C}_j aE^j_2$. \\

 Observe that,
\begin{eqnarray}\label{equa-e1j}
        \sum_{\substack{j=1 \\ j\neq m}}^{[a^{-1}]}\bar{C}_j aE^j_2
       &=&\sum_{\substack{j=1 \\ j\neq m\\ \Omega_j\in N_{\Omega_m}}}^{[a^{-1}]}\bar{C}_j aE^j_2
                                                    +\sum_{\substack{j=1 \\ \\ j\neq m\\ \Omega_j\in F_{\Omega_m}}}^{[a^{-1}]}\bar{C}_j aE^j_2.                                           
\end{eqnarray}

\bigskip
Now by writing $f'(z_m,y):=\Phi_\kappa(z_m,y) Y(y)$. For $z_l\in\Omega_j,\,j\neq m$, using Taylor series, we can write
  $$f'(z_m,z_j)-f'(z_m,z_l)=(z_j-z_l)R'(z_m;z_j,z_l),$$
  with
 \begin{eqnarray}\label{taylorremind1'}
   R'(z_m;z_j,z_l)
   &=&\int_0^1\nabla_y f'(z_m,z_j-\beta(z_j-z_l))\,d\beta.
  \end{eqnarray}
By doing the computations similar to the ones we have performed in (\ref{taylorremind1}-\ref{taylorremind-effect}) and by using Lemma \ref{invertibility-of-VC}, we obtain
\begin{eqnarray}
\left\vert\sum_{\substack{j=1 \\ j\neq m\\ \Omega_j\in N_{\Omega_m}}}^{[a^{-1}]}\bar{C}_j aE^j_2\right\vert&\leq&\frac{c_1 2^\frac{1}{3}}{6}(K_{max}-1)\mathbf{C}\frac{a^\alpha}{{d}}\,a^\frac{1}{3}\,\left(\left[\frac{2^\frac{4}{3}}{d}+\kappa\right]\right.
  \left.+  c_5 \right)   \label{este2j1}\\
\left\vert\sum_{\substack{j=1 \\ j\neq m\\ \Omega_j\in F_{\Omega_m}}}^{[a^{-1}]}\bar{C}_j aE^j_2\right\vert&\leq&\frac{c_1}{2\pi}(K_{max}-1)
\mathbf{C}\frac{(a^{-1}-1)}{{{a}^\frac{\alpha}{3}}}\left(\left[\frac{2}{{a}^\frac{\alpha}{3}}+\kappa\right]+c_5\right)\,a^\frac{1}{3}\,a   \label{este2j2}.
\end{eqnarray}
One can easily see that,
\begin{eqnarray}
\vert\bar{C}_m a E^j_1\vert&\leq&\frac{c_1(K_{max}-1)\mathbf{C}}{4\pi}\frac{a}{d}.   \label{este1j}
\end{eqnarray}
\bigskip

Now, substitution of \eqref{integralonomega} in \eqref{fracqcfracmain-effect-int-1} and using the estimates 
\eqref{integralonomega-subelements-abs}, \eqref{integralonomega-subelements-abs-F} and \eqref{estmatemthint-effe-acc}  associated to $A$ and the estimates 
\eqref{este2j1}, \eqref{este2j2}, \eqref{este1j} associated to $B$  gives us

\begin{eqnarray}
 Y(z_m) +\sum_{\substack{j=1 \\ j\neq m}}^{M} \Phi_\kappa(z_m,z_j)\bar{C}_j Y(z_j) a
  &=&-U^{i}(z_m, \theta)+O\left(\frac{3}{8\pi}c_1 K_{max}\mathbf{C}\left(\frac{4}{3\pi} \right)^\frac{1}{3}a^\frac{2}{3}\right)\nonumber\\
 &&+O\left(\frac{c_1K_{max}\mathbf{C} 2^\frac{1}{3}}{6}{a}^{\alpha}\frac{a^\frac{1}{3}}{d}\left[\left(\frac{2^\frac{4}{3}}{d}+\kappa\right)+c_5\right]\right).\nonumber\\
  &&+O\left(\frac{c_1K_{max}\mathbf{C}}{2\pi}(a^{-1}-1)\frac{a^\frac{4}{3}}{{a}^\frac{\alpha}{3}}\left[\left(\frac{2}{{a}^\frac{\alpha}{3}}+\kappa\right)+c_5\right]\right)\nonumber\\
   &&+O\left(\frac{c_1(K_{max}-1)\mathbf{C} }{4\pi}\frac{a}{d}\right).\label{fracqcfracmain-effect-int-3}
\end{eqnarray}

Taking the difference between \eqref{fracqcfracmain-effect} and \eqref{fracqcfracmain-effect-int-3} produces the algebraic system
\begin{eqnarray}\label{fracqcfracmain-effect-int-4}
 (Y_m-Y(z_m)) +\sum_{\substack{j=1 \\ j\neq m}}^{M} \Phi_\kappa(z_m,z_j)\bar{C}_j (Y_j-Y(z_j)) a 
 &=&O\left(\frac{3}{8\pi}c_1 K_{max}\mathbf{C}\left(\frac{4}{3\pi} \right)^\frac{1}{3}a^\frac{2}{3}\right)\\
 &&\hspace{-2cm}+O\left(\frac{c_1 K_{max}\mathbf{C} 2^\frac{1}{3}}{6}{a}^{\alpha}\frac{a^\frac{1}{3}}{d}\left[\left(\frac{2^\frac{4}{3}}{d}+\kappa\right)+c_5\right]\right)\nonumber\\
  &&\hspace{-2cm}+O\left(\frac{c_1 K_{max}\mathbf{C}}{2\pi}(a^{-1}-1)\frac{a^\frac{4}{3}}{{a}^\frac{\alpha}{3}}\left[\left(\frac{2}{{a}^\frac{\alpha}{3}}+\kappa\right)+c_5\right]\right)\nonumber\\
     &&\hspace{-2cm}+O\left(\frac{c_1(K_{max}-1)\mathbf{C} }{4\pi}\frac{a}{d}\right).\nonumber
\end{eqnarray}

Comparing this system with \eqref{fracqcfracmain-effect} and by using Lemma \ref{Mazyawrkthm-effect}, we obtain the estimate

\begin{eqnarray}\label{mazya-fnlinvert-small-ac-3-effect-dif}
 \sum_{m=1}^{M}(Y_m-Y(z_m))&=&O\left(\left(M\frac{3}{8\pi}c_1 K_{max}\mathbf{C}\left(\frac{4}{3\pi} \right)^\frac{1}{3}a^\frac{2}{3}\right)\right.\nonumber\\
 &&\left.+M\left(\frac{c_1 K_{max}\mathbf{C} 2^\frac{1}{3}}{6}{a}^{\alpha}\frac{a^\frac{1}{3}}{d}\left[\left(\frac{2^\frac{4}{3}}{d}+\kappa\right)+c_5\right]\right)\right.\nonumber\\
  &&\left.+M\left(\frac{c_1 K_{max}\mathbf{C}}{2\pi}(a^{-1}-1)\frac{a^\frac{4}{3}}{{a}^\frac{\alpha}{3}}\left[\left(\frac{2}{{a}^\frac{\alpha}{3}}+\kappa\right)+c_5\right]\right)\right.\nonumber\\
  &&\left.+M\left(\frac{c_1(K_{max}-1)\mathbf{C} }{4\pi}\frac{a}{d}\right)\right).
\end{eqnarray}

For the special case $d=a^t,\,M=O(a^{-s})$ with  $t,s>0$, we have the following approximation of the far-field from the Foldy-Lax asymptotic expansion \eqref{x oustdie1 D_m farmain-recent-near} and from the definitions $Y_m:=\frac{Q_m}{C_m}$ and $C_m:=\bar{C}_m a$, for $m=1,\dots,M$:
\begin{eqnarray}\label{x oustdie1 D_m farmain-recent**-effect}
U^\infty(\hat{x},\theta) &=&\sum_{j=1}^{M}e^{-i\kappa\hat{x}\cdot z_j}\bar{C}_jY_ja\\ \nonumber
&&+O\left(a^{2-s}\hspace{-.03cm}+\hspace{-.03cm}a^{3-s-5t+3t\alpha}\hspace{-.03cm}+\hspace{-.03cm}a^{4-s-9t+6t\alpha}\hspace{-.03cm}+\hspace{-.03cm}a^{3-2s-2t\alpha}\hspace{-.03cm}+\hspace{-.03cm}a^{4-3s-3t\alpha}\hspace{-.03cm}+\hspace{-.03cm}a^{4-2s-5t+2t\alpha}\right).
\end{eqnarray}
Consider the far-field of type:
\begin{eqnarray}\label{acoustic-farfield-effect}
U^\infty_{C^M}(\hat{x},\theta) &=& \int_{\Omega} e^{-i\kappa\hat{x}\cdot{y}}{K^{M}}(y){C^{M}}(y) Y(y) dy.
\end{eqnarray}
Taking the difference between \eqref{acoustic-farfield-effect} and \eqref{x oustdie1 D_m farmain-recent**-effect} gives us:
\begin{eqnarray}\label{acoustic-difference-farfield-effect}
 U^\infty_{C^M}(\hat{x},\theta)-U^\infty(\hat{x},\theta) &=& \int_{\Omega} e^{-i\kappa\hat{x}\cdot{y}}{K^{M}}(y){C^{M}}(y) Y(y) dy- \sum_{j=1}^{M}e^{-i\kappa\hat{x}\cdot z_j}\bar{C}_jY_ja\nonumber\\ \nonumber
 &&\hspace{0cm}+O\left(a^{2-s}\hspace{-.03cm}+\hspace{-.03cm}a^{3-s-5t+3t\alpha}\hspace{-.03cm}+\hspace{-.03cm}a^{4-s-9t+6t\alpha}\hspace{-.03cm}+\hspace{-.03cm}a^{3-2s-2t\alpha}\hspace{-.03cm}+\hspace{-.03cm}a^{4-3s-3t\alpha}\hspace{-.03cm}+\hspace{-.03cm}a^{4-2s-5t+2t\alpha}\right)
\\ \nonumber &=& \sum_{j=1}^{[a^{-1}]}\int_{\Omega_j} e^{-i\kappa\hat{x}\cdot{y}}{K^{M}}(y){C^{M}}(y) Y(y)dy -\sum_{j=1}^{[a^{-1}]} \sum_{\substack{l=1 \\ z_l \in \Omega_j}}^{\textcolor{black}{[K^M(z_j)]}}e^{-i\kappa\hat{x}\cdot z_l}\bar{C}_lY_la\nonumber\\ \nonumber
 &&\hspace{0cm}+O\left(a^{2-s}\hspace{-.03cm}+\hspace{-.03cm}a^{3-s-5t+3t\alpha}\hspace{-.03cm}+\hspace{-.03cm}a^{4-s-9t+6t\alpha}\hspace{-.03cm}+\hspace{-.03cm}a^{3-2s-2t\alpha}\hspace{-.03cm}+\hspace{-.03cm}a^{4-3s-3t\alpha}\hspace{-.03cm}+\hspace{-.03cm}a^{4-2s-5t+2t\alpha}\right)
\\ \nonumber &=& \sum_{j=1}^{[a^{-1}]}{K^{M}}(z_j)\bar{C}_j\int_{\Omega_j} \left[e^{-i\kappa\hat{x}\cdot{y}} Y(y) - e^{-i\kappa\hat{x}\cdot z_j}Y(z_j)\right]dy\nonumber\\ \nonumber
\\ \nonumber &&+ \sum_{j=1}^{[a^{-1}]}\bar{C}_ja \left[\sum_{\substack{l=1 \\ z_l \in \Omega_j}}^{\textcolor{black}{[K^M(z_j)]}}\left(e^{-i\kappa\hat{x}\cdot z_j} Y(z_j)-e^{-i\kappa\hat{x}\cdot z_l}Y(z_l)\right)+\sum_{\substack{l=1 \\ z_l \in \Omega_j}}^{\textcolor{black}{[K^M(z_j)]}} e^{-i\kappa\hat{x}\cdot z_l} \left(Y(z_l)-Y_l\right)\right]\nonumber\\ \nonumber
 &&\hspace{0cm}+O\left(a^{2-s}\hspace{-.03cm}+\hspace{-.03cm}a^{3-s-5t+3t\alpha}\hspace{-.03cm}+\hspace{-.03cm}a^{4-s-9t+6t\alpha}\hspace{-.03cm}+\hspace{-.03cm}a^{3-2s-2t\alpha}\hspace{-.03cm}+\hspace{-.03cm}a^{4-3s-3t\alpha}\hspace{-.03cm}+\hspace{-.03cm}a^{4-2s-5t+2t\alpha}\right)
 \\ \nonumber &=& \sum_{j=1}^{[a^{-1}]}\int_{\Omega_j}{K^{M}}(z_j)\bar{C}_j \left[e^{-i\kappa\hat{x}\cdot{y}} Y(y) - e^{-i\kappa\hat{x}\cdot z_j}Y(Z_j)\right]dy\nonumber\\
 &&+ \sum_{j=1}^{[a^{-1}]}\bar{C}_ja \sum_{\substack{l=1 \\ z_l \in \Omega_j}}^{\textcolor{black}{[K^M(z_j)]}}\left(e^{-i\kappa\hat{x}\cdot z_j} Y(z_j)-e^{-i\kappa\hat{x}\cdot z_l}Y(z_l)\right)+\sum_{j=1}^{M}e^{-i\kappa\hat{x}\cdot z_j}\bar{C}_ja \left[Y(z_j)-Y_j\right]\nonumber\\ \nonumber
 &&\hspace{0cm}+O\left(a^{2-s}\hspace{-.03cm}+\hspace{-.03cm}a^{3-s-5t+3t\alpha}\hspace{-.03cm}+\hspace{-.03cm}a^{4-s-9t+6t\alpha}\hspace{-.03cm}+\hspace{-.03cm}a^{3-2s-2t\alpha}\hspace{-.03cm}+\hspace{-.03cm}a^{4-3s-3t\alpha}\hspace{-.03cm}+\hspace{-.03cm}a^{4-2s-5t+2t\alpha}\right)\\
 &\substack{= \\ \eqref{mazya-fnlinvert-small-ac-3-effect-dif} }& \sum_{j=1}^{[a^{-1}]}{K^{M}}(z_j)\bar{C}_j\int_{\Omega_j} \left[e^{-i\kappa\hat{x}\cdot{y}} Y(y) - e^{-i\kappa\hat{x}\cdot z_j}Y(z_j)\right]dy\\
 &&+ \sum_{j=1}^{[a^{-1}]}\bar{C}_ja \sum_{\substack{l=1 \\ z_l \in \Omega_j}}^{\textcolor{black}{[K^M(z_j)]}}\left(e^{-i\kappa\hat{x}\cdot z_j} Y(z_j)-e^{-i\kappa\hat{x}\cdot z_l}Y(z_l)\right)\nonumber\\
&&\qquad+\mathbf{C}\,a\,O\left(\left(M\frac{3}{8\pi}c_1K_{max}\mathbf{C}\left(\frac{4}{3\pi} \right)^\frac{1}{3}a^\frac{2}{3}\right)+M\left(\frac{c_1(K_{max}-1)\mathbf{C} }{4\pi}\frac{a}{d}\right)\right.\nonumber\\
 &&\qquad\left.+M\left(\frac{c_1K_{max}\mathbf{C} 2^\frac{1}{3}}{6}{a}^{\alpha}\frac{a^\frac{1}{3}}{d}\left[\left(\frac{2^\frac{4}{3}}{d}+\kappa\right)+c_5\right]\right)\right.\nonumber\\
  &&\qquad\left.+M\left(\frac{c_1K_{max}\mathbf{C}}{2\pi}(a^{-1}-1)\frac{a^\frac{4}{3}}{{a}^\frac{\alpha}{3}}\left[\left(\frac{2}{{a}^\frac{\alpha}{3}}+\kappa\right)+c_5\right]\right)\right)\nonumber\\ \nonumber
   &&\hspace{0cm}+O\left(a^{2-s}\hspace{-.03cm}+\hspace{-.03cm}a^{3-s-5t+3t\alpha}\hspace{-.03cm}+\hspace{-.03cm}a^{4-s-9t+6t\alpha}\hspace{-.03cm}+\hspace{-.03cm}a^{3-2s-2t\alpha}\hspace{-.03cm}+\hspace{-.03cm}a^{4-3s-3t\alpha}\hspace{-.03cm}+\hspace{-.03cm}a^{4-2s-5t+2t\alpha}\right).
 \end{eqnarray}
Now, let us estimate the  difference $\sum_{j=1}^{[a^{-1}]}{K^{M}}(z_j)\bar{C}_j\int_{\Omega_j} \left[e^{-i\kappa\hat{x}\cdot{y}} Y(y) - 
e^{-i\kappa\hat{x}\cdot z_j}Y(z_j)\right]dy$.
\begin{itemize}
\item Write, $f_1(y)=e^{-i\kappa\hat{x}\cdot{y}} Y(y)$. Using Taylor series, we can write
  $$f_1(y)-f_1(z_j)=(y-z_j)R_j(y),$$
  with 
  \begin{eqnarray}
   R_j(y)
   &=&\int_0^1\nabla_y f_1(y-\beta(y-z_j))\,d\beta\nonumber\\
   &=&\int_0^1\nabla_y \left[e^{-i\kappa\hat{x}\cdot(y-\beta(y-z_j))} Y(y-\beta(y-z_j))\right]\,d\beta\nonumber\\
   &=&\int_0^1\left[\nabla_ye^{-i\kappa\hat{x}\cdot(y-\beta(y-z_j))}\right] Y(y-\beta(y-z_j))\,d\beta\nonumber\\
   &&+\int_0^1e^{-i\kappa\hat{x}\cdot(y-\beta(y-z_j))}\left[\nabla_y Y(y-\beta(y-z_j))\right]\,d\beta.
  \end{eqnarray}

  We have $\nabla_ye^{-i\kappa\hat{x}\cdot{y}}=-i\kappa\hat{x}e^{-i\kappa\hat{x}\cdot{y}}$. 
 It  gives us
 \begin{eqnarray}\label{taylorremind-effect-singvariable}
 \vert R_j(y) \vert
   &\leq&\left(\kappa \int_0^1 |Y(y-\beta(y-z_j))|\, d\beta\,+\,\int_0^1\vert\nabla_y Y(y-\beta(y-z_j))\vert\,d\beta\right).
  \end{eqnarray}
  Using  \eqref{taylorremind-effect-singvariable} we get the estimate
 \begin{eqnarray}\label{integralonomega-subelements-abs-effect}
  \left\vert\sum_{j=1}^{[a^{-1}]}{K^{M}}(z_j)\bar{C}_j\int_{\Omega_j} \left[e^{-i\kappa\hat{x}\cdot{y}}(y) Y(y) - e^{-i\kappa\hat{x}\cdot z_j}Y(z_j)\right]dy\right\vert&& \nonumber\\
  &\hspace{-15cm}\leq&\hspace{-7cm}\sum_{j=1}^{[a^{-1}]}{K^{M}}(z_j)\bar{C}_j \left(\kappa \int_{\Omega_j}\vert y-z_j\vert\int_0^1 |Y(y-\beta(y-z_j))|\, d\beta\,dy\right)\, \nonumber\\
  &\hspace{-15cm}+&\hspace{-7cm}\,\sum_{j=1}^{[a^{-1}]}{K^{M}}(z_j)\bar{C}_j \left(\int_{\Omega_j}\vert y-z_j\vert\int_0^1\vert\nabla_y Y(y-\beta(y-z_j))\vert\,d\beta\,dy\right) \nonumber\\
  &\hspace{-15cm}\substack{\leq\\ \mbox{As in \eqref{integralonomega-subelements-abs}}}&\hspace{-6.5cm}\sum_{j=1}^{[a^{-1}]}{K^{M}}(z_j)\bar{C}_jc_1\, a\,a^\frac{1}{3}\,\left(\kappa +c_5\right).\nonumber\\  
    &\hspace{-15cm}\leq&\hspace{-7cm} K_{max} \mathbf{C}c_1\left(\kappa +c_5\right)\,a^\frac{1}{3}.
 \end{eqnarray}
 \end{itemize}
  
 In the similar way, we can also show that,
 \begin{eqnarray}\label{integralonomega-subelements-abs-effect-1}
 \left\vert \sum_{j=1}^{[a^{-1}]}\bar{C}_ja \sum_{\substack{l=1 \\ z_l \in \Omega_j}}^{\textcolor{black}{[K^M(z_j)]}}\left(e^{-i\kappa\hat{x}\cdot z_j} Y(z_j)-e^{-i\kappa\hat{x}\cdot z_l}Y(z_l)\right)\right\vert
 &\leq& (K_{max}-1) \mathbf{C}c_1\left(\kappa +c_5\right)\,a^\frac{1}{3}.
 \end{eqnarray}
 Using the estimates \eqref{integralonomega-subelements-abs-effect} and \eqref{integralonomega-subelements-abs-effect-1} in \eqref{acoustic-difference-farfield-effect}, we obatin
 \begin{eqnarray}\label{acoustic-difference-farfield-effect-1}
 U^\infty_{C^M}(\hat{x},\theta)-U^\infty(\hat{x},\theta) 
 &=& O\left(K_{max}a^\frac{1}{3} \mathbf{C}c_1\left(\kappa +c_5\right)\right)\nonumber\\
&&+O\left(M\,a\,a^\frac{2}{3}\left(\frac{3}{8\pi}c_1K_{max}\mathbf{C}^2\left(\frac{4}{3\pi} \right)^\frac{1}{3}\right)+Ma\left(\frac{c_1(K_{max}-1)\mathbf{C}^2 }{4\pi}\frac{a}{d}\right)\right.\nonumber\\
 &&\qquad\left.+Ma\left(\frac{c_1K_{max}\mathbf{C}^2 2^\frac{1}{3}}{6}{a^{\alpha}}\frac{a^\frac{1}{3}}{d}\left[\left(\frac{2^\frac{4}{3}}{d}+\kappa\right)+c_5\right]\right)\right.\nonumber\\
  &&\qquad\left.+Ma\left(\frac{c_1K_{max}\mathbf{C}^2}{2\pi}(a^{-1}-1)\frac{a^\frac{4}{3}}{{a^{\frac{\alpha}{3}}}}\left[\left(\frac{2}{{a^{\frac{\alpha}{3}}}}+\kappa\right)+c_5\right]\right)\right)\nonumber\\ \nonumber
   &&\hspace{0cm}+O\left(a^{2-s}\hspace{-.03cm}+\hspace{-.03cm}a^{3-s-5t+3t\alpha}\hspace{-.03cm}+\hspace{-.03cm}a^{4-s-9t+6t\alpha}\hspace{-.03cm}+\hspace{-.03cm}a^{3-2s-2t\alpha}\hspace{-.03cm}+\hspace{-.03cm}a^{4-3s-3t\alpha}\hspace{-.03cm}+\hspace{-.03cm}a^{4-2s-5t+2t\alpha}\right) 
 \nonumber\\
&\substack{=\\{M=O(a^{-1})\,}\\{d=a^t,t<1\,}}&O\left( a^\frac{1}{3}K_{max}\mathbf{C}c_1\left(\kappa_{\max} +c_5\right)\right) \nonumber\\
&&+O\left(a^\frac{2}{3}\frac{3}{8\pi}c_1K_{max}\mathbf{C}^2\left(\frac{4}{3\pi} \right)^\frac{1}{3}+a^{1-t}\frac{c_1(K_{max}-1)\mathbf{C}^2 }{4\pi}\right.\nonumber\\
 &&\qquad\left.+\frac{c_1K_{max}\mathbf{C}^2 2^\frac{1}{3}}{6}a^{\alpha-2t}a^\frac{1}{3}\left[\left(2^\frac{4}{3}+d_{\max}\kappa_{\max}\right)+c_5d_{\max}\right]\right.\nonumber\\
  &&\qquad\left.+\frac{c_1K_{max}\mathbf{C}^2}{2\pi}\frac{a^{\frac{1-\alpha}{3}}}{2}\left[\left(2a^{-\frac{\alpha}{3}}+\kappa_{\max}\right)+c_5\right]\right)\nonumber\\ 
   &&\hspace{0cm}+O\left(a\hspace{-.03cm}+\hspace{-.03cm}a^{2-5t+3t\alpha}\hspace{-.03cm}+\hspace{-.03cm}a^{3-9t+6t\alpha}\hspace{-.03cm}+\hspace{-.03cm}a^{1-2t\alpha}\hspace{-.03cm}+\hspace{-.03cm}a^{1-3t\alpha}\hspace{-.03cm}+\hspace{-.03cm}a^{2-5t+2t\alpha}\right).\nonumber\\
      &=&O\left( a^\frac{1}{3}\hspace{-.03cm}+\hspace{-.03cm}a^\frac{2}{3}\hspace{-.03cm}+\hspace{-.03cm}a^{1-t}+a^{\alpha-2t+\frac{1}{3}}\hspace{-.03cm}+\hspace{-.03cm}
   a^{\frac{1-2\alpha}{3}}\hspace{-.03cm}+\hspace{-.03cm}a^{\frac{1-\alpha}{3}}\hspace{-.03cm}+\hspace{-.03cm}a\right.\nonumber\\
 &&\qquad\left.\hspace{-.03cm}+\hspace{-.03cm}a^{2-5t+3t\alpha}+a^{3-9t+6t\alpha}\hspace{-.03cm}+\hspace{-.03cm}a^{1-2t\alpha}\hspace{-.03cm}+\hspace{-.03cm}a^{1-3t\alpha}\hspace{-.03cm}+\hspace{-.03cm}a^{2-5t+2t\alpha}\right)\nonumber\\
    &=&O\left( a^{1-t}+a^{\alpha-2t+\frac{1}{3}}\hspace{-.03cm}+\hspace{-.03cm}
   a^{\frac{1-2\alpha}{3}}+a^{3-9t+6t\alpha}\hspace{-.03cm}+\hspace{-.03cm}a^{2-5t+2t\alpha}\hspace{-.03cm}+\hspace{-.03cm}a^{1-3t\alpha}\right).
  \end{eqnarray}

 Since $Vol(\Omega)$ is of order $a^{-1}(\frac{a}{2}+\frac{d}{2})^3$, then we observed that ${d\leq O(a^{\frac{1}{3}}})$, otherwise this volume exploses as $a\rightarrow 0$. 
 In particular for $d$ of the form $a^t$, we should have {$t\geq\frac{1}{3}$}. From the above we should have,

  \begin{eqnarray*}
  t\geq\frac{1}{3}; \qquad 1-t>0; \qquad \alpha-2t+\frac{1}{3}>0; \qquad   \frac{1-2\alpha}{3}>0;     \\
      3-9t+6t\alpha>0;\qquad  1-3t\alpha>0;\qquad 2-5t+2t\alpha>0;
  \end{eqnarray*}
  
which further reduces to {
\begin{eqnarray*}
 \frac{1}{3}\leq t<1;  &\qquad&0<\alpha<\frac{1}{2}; \qquad t\alpha<\frac{1}{3};\\
 2-5t+2t\alpha>0;&\Longrightarrow& t< \frac{2}{5-2\alpha}\,\rightarrow\,\alpha<\frac{5}{2};\\
 1-6t+3\alpha>0;&\Longrightarrow& t< \frac{3\alpha+1}{6}\,\rightarrow\,{\alpha>-\frac{1}{3}};\\
 3-9t+6t\alpha>0;&\Longrightarrow& t< \frac{1}{3-2\alpha}\,\rightarrow\,\alpha<\frac{3}{2}.\\
\end{eqnarray*}
Hence for {$\frac{1}{3}\leq{t}<1$}, we have 
{
\begin{eqnarray*}
 {0<\alpha<\frac{1}{2}};\qquad   \frac{1-2\alpha}{3}>0;\qquad t\alpha<\frac{1}{3};\qquad2-5t+2t\alpha>0;\qquad
  1-6t+3\alpha>0;\qquad
 3-9t+6t\alpha>0;
\end{eqnarray*}
}
\begin{figure}[htp]
\centering
\includegraphics[width=15.5cm,height =10.5cm,natwidth=610,natheight=642]{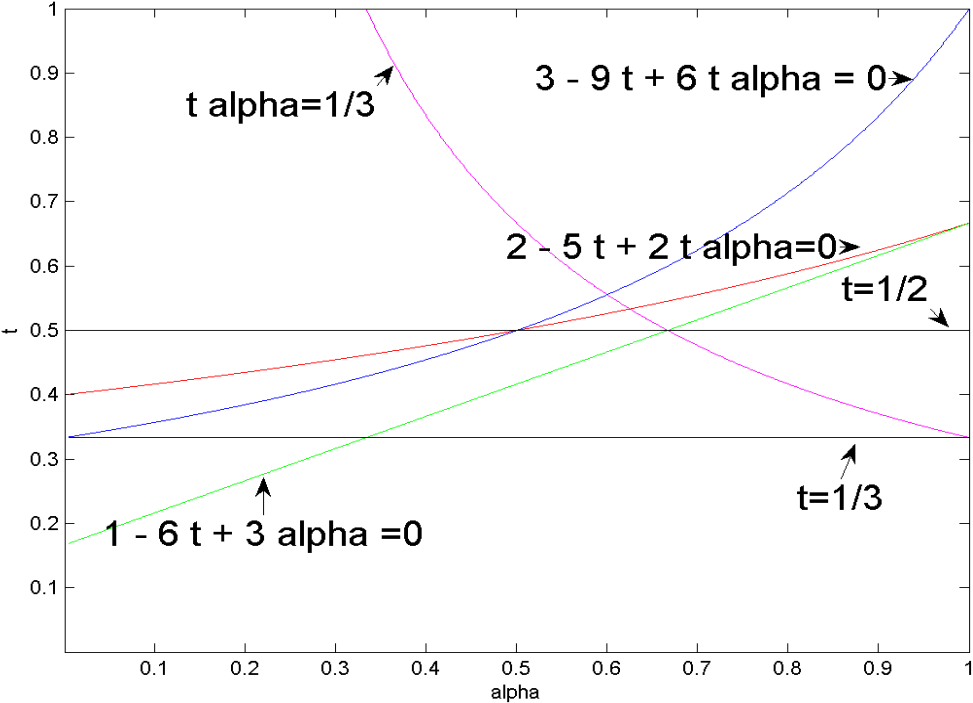}
\caption{Relation between $\alpha$ and $t$.}\label{fig:Comparison-alphat2}
\end{figure}

By solving these inequalities, also see figure \ref{fig:Comparison-alphat2}, we obtain {$\frac{1}{3}\leq {t}<\frac{5}{12}$} and {$\frac{1}{3}< {\alpha}<\frac{1}{2}$}, in particular {$\frac{1}{2^\frac{3}{2}}\leq {\alpha}<\frac{1}{2}$}. Precisely,

\begin{eqnarray}\label{acoustic-difference-farfield-effect-1**1-1}
 U^\infty_{C^M}(\hat{x},\theta)-U^\infty(\hat{x},\theta)
      &=&O\left( a^{\frac{1-6t+3\alpha}{3}}+a^{\frac{1-2\alpha}{3}}\right).
  \end{eqnarray}
  Equaling the exponents in the right hand side, i.e. $\frac{1-6t+3\alpha}{3}=\frac{1-2\alpha}{3}$, we deduce that $\alpha=\frac{6}{5}t$. Hence
\begin{eqnarray}\label{acoustic-difference-farfield-effect-1**1-2}
 U^\infty_{C^M}(\hat{x},\theta)-U^\infty(\hat{x},\theta)=O\left( a^{\frac{1}{3}\left(1-\frac{12}{5}t\right)}\right).
  \end{eqnarray}

\begin{remark}\label{Remark-of-the-final-estimate}
For the setting of each $\Omega_m$ containing only one obstacle $D_m$, $m=1,\dots,M$, i.e. $M=[a^{-1}]$  we can observe that $t=\frac{1}{3}$ and $\frac{1}{3}<\alpha<\frac{1}{2}$. In this case we can derive the estimate as
\begin{eqnarray}\label{acoustic-difference-farfield-effect-1**1}
 U^\infty_{C^M}(\hat{x},\theta)-U^\infty(\hat{x},\theta)
      &=&O\left( a^{\frac{3\alpha-1}{3}}+a^{\frac{1-2\alpha}{3}}\right).
  \end{eqnarray}
It can also be written as 
\begin{eqnarray}\label{acoustic-difference-farfield-effect-1***1}
 U^\infty_{C^M}(\hat{x},\theta)-U^\infty(\hat{x},\theta) 
      &=&\left\{\begin{array}{ccc}
      O\left( a^{\frac{3\alpha-1}{3}}\right),& \mbox{ for } \frac{1}{3}<\alpha\leq\frac{2}{5},\\ \ ~ \ \\
      O\left( a^{\frac{1}{15}}\right),& \mbox{ for } \alpha=\frac{2}{5},\\ \ ~ \ \\
      O\left(a^{\frac{1-2\alpha}{3}}\right),&\mbox{ for } \frac{2}{5}\leq\alpha<\frac{1}{2}.\\     
         \end{array}
         \right.
\end{eqnarray}
In this case ( $s=1,t=\frac{1}{3}$ ) it is clear that the best estimate is attained for $\alpha=\frac{2}{5}$ and hence the following error estimate holds
   \begin{eqnarray}\label{final-estimate}
 U^\infty_{C^M}(\hat{x},\theta)-U^\infty(\hat{x},\theta) 
      &=&O\left( a^{\frac{1}{15}}\right).
  \end{eqnarray}
  
  It can also be observed that estimate \eqref{acoustic-difference-farfield-effect-1**1-1} leads to the estimate \eqref{final-estimate} for $t=\frac{1}{3}$.
\end{remark}

 \subsection{Case when the obstacles are arbitrarily distributed}\label{arbitrarely-distributed}
 In this case, we take $N=[a^{-1}]$ and $N_j=1, j=1, ..., N$ in the way we divide $\Omega$, i.e. $\Omega:=\cup^{[a^{-1}]}_{j=1}\Omega_j$, $\Omega_j$'s are disjoint,
 see the beginning of Subsection \ref{subsection-piecewise-constant}.
 Hence, due to the analysis in the last subsection, we end up with the following approximation
 \begin{equation}\label{piecewise-constant-capacitances}
  U^\infty(\hat{x},\theta)=U^\infty_{a}(\hat{x},\theta) +o(a^\frac{1}{15}),\; \; a\rightarrow 0
  \end{equation}
where $U^\infty_{a}(\hat{x},\theta)$ is the farfield corresponding to the following scattering problem
\begin{equation}
(\Delta + \kappa^{2}-K_a\bold{C}_a)U_{a}^{t}=0 \mbox{ in }\mathbb{R}^{3},\label{piecewise-1}
\end{equation}
\begin{equation}
U_{a}^{t}=U_{a}^s +e^{i\kappa x\cdot \theta}  
\end{equation}
\begin{equation}
\frac{\partial U_{a}^{s}}{\partial |x|}-i\kappa U_{a}^{s}=o\left(\frac{1}{|x|}\right), |x|\rightarrow\infty, ~(\text{S.R.C}) \label{radiation-piecewise-1}
\end{equation}
 and $\bold{C}_a$ is the potential defined as follows: $\bold{C}_a=C_j$ in $\Omega_j, \; j:=1, ..., M$ and $\bold{C}_a=0$ in $\mathbb{R}^{3}\setminus{\overline{\Omega}}$.
 We have similar properties for $K_a$.
 \bigskip
 
 We know that the solution of this last scattering problem satisfies the Lippmann-Schwinger equation
 \begin{equation}\label{Lip-Sc-piecewise-constant}
  U_{a}^{t}(x) +\int_{\Omega}K_a\bold{C}_a(y)\Phi(x, y)U_{a}^{t}(y)dy = e^{i\kappa x\cdot \theta}, \; x\in \Omega.
 \end{equation}
 In addition, we know that the function $\bold C_a$, defined from the sequence $(C_j)^M_{j=1}$, recalling that $M=[a^{-1}]$, is bounded as function of $a\in (0, 1)$ as $a\rightarrow 0$, i.e. $\bold C_a$ is bounded in 
 $L^\infty(\Omega)$. 
 Indeed, the capacitances of the obstacles $B_j$, i.e. $C_j$ are bounded by their Lipschitz constants, see \cite{C-S:2014}, and we assumed that these Lipschitz constants are uniformly bounded.
 Hence $\bold C_a$ is bounded in $L^{2}(\Omega)$ and then there exists a function $\bold C_0$ in $L^2(\Omega)$ (actually in every $L^p(\Omega)$) such that $\bold C_a$ converges weakly to $\bold C_0$
 in $L^{2}(\Omega)$. Now, since $K$ is continuous hence $K_a$ converges to $K$ in $L^\infty(\Omega)$ and hence in $L^2(\Omega)$. 
 Then we can show that $K_a \bold C_a$ converges to $K \bold C_0$ in $L^2(\Omega)$. 
 \bigskip
 
 Since $K\bold C_a$ is bounded in $L^{\infty}(\Omega)$, then from the invertibility of the Lippmann-Schwinger equation and the mapping properties of the Poisson potential, 
 see Lemma \ref{invertibility-of-VC},
 we deduce that $\Vert U_{a}^{t}\Vert_{H^2(\Omega)}$ is bounded and in particular, up to a sub-sequence, $U_{a}^{t}$ tends to $U_{0}^{t}$ in $L^2(\Omega)$. 
 From the convergence of $K_a\bold C_a$ to $K\bold C_0$
 and the one of $U_{a}^{t}$ to $U_{0}^{t}$ and (\ref{Lip-Sc-piecewise-constant}), 
 we derive the following equation satisfied by 
 $U_{0}^{t}(x)$
 $$
 U_{0}^{t}(x) +\int_{\Omega}K\bold{C}_0(y)\Phi(x, y)U_{0}^{t}(y)dy=e^{i\kappa x\cdot \theta}\; \mbox{ in } \Omega.
 $$
This is of course the Lippmann-Schwinger equation corresponding to the scattering problem
\begin{equation}
(\Delta + \kappa^{2}-K \bold{C}_0)U_{0}^{t}=0 \mbox{ in }\mathbb{R}^{3},\label{piecewise-2}
\end{equation}
\begin{equation}
U_{0}^{t}=U_{0}^s +e^{i\kappa x\cdot \theta},  
\end{equation}
\begin{equation}
\frac{\partial U_{0}^{s}}{\partial |x|}-i\kappa U_{0}^{s}=o\left(\frac{1}{|x|}\right), |x|\rightarrow\infty, ~(\text{S.R.C}). \label{radiation-piecewise-2}
\end{equation}
 As the corresponding farfields are of the form
 $$
 U_0^{\infty}(\hat{x}, \theta)=\int_{\Omega}e^{-i\kappa \hat{x}\cdot y}K\bold{C}_0(y)U_{0}^{t}(y)dy
 $$
 and the ones of $U^t_{a}$ are of the form
 $$
 U_a^{\infty}(\hat{x}, \theta)=\int_{\Omega}e^{-i\kappa \hat{x}\cdot y}K_a\bold{C}_a(y)U_{a}^{t}(y)dy
 $$
 we deduce that
 $$
 U^\infty_{a}(\hat{x},\theta)-U^\infty_{0}(\hat{x},\theta)=o(1),\; a\rightarrow 0, \mbox{ uniformly in terms of } \hat x , \theta\; \in \mathbb{S}^{2}.
 $$

 \subsection{Case when $K$ is H$\ddot{\mbox{o}}$lder continuous}\label{smoothly-distributed}
 
 Finally assume that $K\in C^{0, \gamma}(\Omega),\; \gamma \in (0, 1]$, then we have the estimate $\Vert K -K_a\Vert_{L^\infty(\Omega)}\leq C a^{\gamma}$, $a<<1$.
 Let $C_0=C$, a constant, in $\Omega$ and $C_0=0$ in $\mathbb{R}^3\setminus \Omega$. Recall that $U_0$ and $U_a$ are solutions of the Lippmann-Schwinger equations 
 $$
 U_0+\int_{\Omega}\Phi(x, y)K\bold{C}_0(y)U_{0}^{t}(y)dy=e^{i \kappa x \cdot \theta}
 $$
 and 
 $$
 U_a+\int_{\Omega}\Phi(x, y)K_a\bold{C}_0(y)U_{a}^{t}(y)dy=e^{i \kappa x \cdot \theta}.
 $$
 From the estimate $\Vert K -K_a\Vert_{L^\infty(\Omega)}\leq C a^{\gamma}$, $a<<1$, we derive the estimate
 \begin{equation}\label{appro-0-a}
  U_0^\infty(\hat{x}, \theta)-U^\infty_a(\hat{x}, \theta)=O(a^\gamma),\; a<<1, \mbox{ uniformly in terms of } \hat x , \theta\; \in \mathbb{S}^{2}.
 \end{equation}
Combining this estimate with (\ref{acoustic-difference-farfield-effect-1**1-2}), we deduce that
\begin{equation}\label{final}
  U^\infty(\hat{x}, \theta)-U_0^\infty(\hat{x}, \theta)=O(a^{\min{\gamma},\; \frac{1}{3}-\frac{4}{5}t}),\; a<<1, \mbox{ uniformly in terms of } \hat x , \theta\; \in \mathbb{S}^{2}.
 \end{equation}

 
 \bibliographystyle{abbrv}

\begin{thebibliography}{10}

\bibitem{A-G-H-H:AMS2005}
S.~Albeverio, F.~Gesztesy, R.~H{\o}egh-Krohn, and H.~Holden.
\newblock {\em Solvable models in quantum mechanics}.
\newblock AMS Chelsea Publishing, Providence, RI, second edition, 2005.

\bibitem{B-L-P:1978} 
 A. Bensoussan; J. L. Lions and G. Papanicolaou. 
 \newblock Asymptotic analysis for periodic structures. 
 \newblock Studies in Mathematics and its Applications, 5. North-Holland Publishing Co., Amsterdam-New York, 1978.  

\bibitem{C-S:2014}
 D. P. Challa; M. Sini.
 \newblock On the justification of the Foldy-Lax approximation for the acoustic scattering by small rigid bodies of arbitrary shapes. 
 \newblock Multiscale Model. Simul. 12 (2014), no. 1, 5508.


\bibitem{C-K:1998}
D.~Colton and R.~Kress.
\newblock {\em Inverse acoustic and electromagnetic scattering theory},
  volume~93 of {\em Applied Mathematical Sciences}.
\newblock Springer-Verlag, Berlin, second edition, 1998.


\bibitem{C-M:1979}
 D. ~Cioranescu and F. ~Murat. 
 \newblock{\em Un terme \'etrange venu d'ailleurs. (French) [A strange term brought from somewhere else] Nonlinear partial differential equations and their applications.} 
 \newblock Coll\`ege de France Seminar, Vol. II (Paris, 1979/1980), pp. 9838, 38990, Res. Notes in Math., 60, Pitman, Boston, Mass.-London, 1982.

 \bibitem{C-M:1997}
 D. ~Cioranescu and F. ~Murat.
 \newblock{ A strange term coming from nowhere}
\newblock Topics in the Mathematical Modelling of Composite Materials.
\newblock Progress in Nonlinear Differential Equations and Their Applications Volume 31, 1997, pp 45-93 
 
\bibitem{LLF:PR1945}
L.~L. Foldy.
\newblock The multiple scattering of waves. {I}. {G}eneral theory of isotropic
  scattering by randomly distributed scatterers.
\newblock {\em Phys. Rev. (2)}, 67:107--119, 1945.

\bibitem{J-K-O:1994}
V. Jikov, S. Kozlov and O. Oleinik.
\newblock {\em Homogenization of differential operators and integral functionals}.
\newblock Springer-Verlag, 1994.


\bibitem{Lax-M:RMP1951}
M.~Lax.
\newblock Multiple scattering of waves.
\newblock {\em Rev. Modern Physics}, 23:287--310, 1951.

\bibitem{Martin:2006}
P.~A. Martin.
\newblock {\em Multiple scattering}, volume 107 of {\em Encyclopedia of
  Mathematics and its Applications}.
\newblock Cambridge University Press, Cambridge, 2006.
\newblock Interaction of time-harmonic waves with $N$ obstacles.

\bibitem{Mclean:2000}
W.~McLean.
\newblock {\em Strongly elliptic systems and boundary integral equations}.
\newblock Cambridge University Press, Cambridge, 2000.

\bibitem{Nachman:1988}
 A. I. Nachman. 
 \newblock Reconstructions from boundary measurements. 
 \newblock Ann. of Math. (2) 128 (1988), no. 3, 53176.


\bibitem{Novikov:1988}  
R. G. Novikov. 
\newblock A multidimensional inverse spectral problem for the equation $\Delta \psi+(v(x)-Eu(x))\psi=0$. 
 \newblock Funct. Anal. Appl. 22 (1988), no. 4, 26372.
 

\bibitem{Ramm:1988}
A. G. Ramm,
\newblock Recovery of the potential from fixed-energy scattering data.
\newblock Inverse Problems 4 (1988), no. 3, 87786.  
 
\bibitem{RAMM:2007}
A.~G. Ramm.
\newblock Many-body wave scattering by small bodies and applications.
\newblock {\em J. Math. Phys.}, 48(10):103511, 29, 2007.


\end{thebibliography}

\end{document}